\documentclass{article}

\usepackage{amsfonts,amsmath,amssymb,amsbsy,amstext,amsthm}
\usepackage[colorlinks,linkcolor=black,citecolor=black,filecolor=black, urlcolor=black]{hyperref}
\usepackage{lineno,url}
\usepackage{mathtools}
\usepackage{nicefrac}
\usepackage[dvipsnames]{xcolor}
\usepackage{multicol}
\usepackage[capitalize]{cleveref}
\crefformat{equation}{(#2#1#3)}
\crefrangeformat{equation}{(#3#1#4) to~(#5#2#6)}
\crefmultiformat{equation}{(#2#1#3)}%
{ and~(#2#1#3)}{, (#2#1#3)}{ and~(#2#1#3)}
\crefmultiformat{fig}{Figures~#2#1#3}{ and~#2#1#3}{, #2#1#3}{ and~#2#1#3}
\crefformat{ex}{Example~#2#1#3}
\crefrangeformat{ex}{Examples~#3#1#4 to~#5#2#6}
\crefmultiformat{ex}{Examples~#2#1#3}%
{ and~#2#1#3}{, #2#1#3}{ and~#2#1#3}
\usepackage{autonum}
\usepackage{accents,enumerate,float,fullpage,verbatim,listings,yfonts}
\usepackage{comment}
\usepackage{tikz,tikz-cd}
\usetikzlibrary{positioning, matrix, shapes,arrows,calc}
\usepackage{subcaption}
\newtheorem{thm}{Theorem}[section]
\newtheorem{lem}[thm]{Lemma}
\newtheorem{cor}[thm]{Corollary}
\newtheorem{prop}[thm]{Proposition}
\newtheorem{conj}[thm]{Conjecture}
\theoremstyle{definition}
\newtheorem{rmk}[thm]{Remark}

\newtheorem{ex}[thm]{Example}

\newcommand{\sumt}{{\langle \theta-1\rangle}}

\newcommand{\kerf}{\ker\varphi_\T}
\newcommand{\imf}{\mathrm{im}\,\varphi_\T}
\newcommand{\imv}{\mathrm{im}\,\nu_{k,d}}

\newcommand{\sip}{SIP} 

\newcommand{\R}{\mathbb{R}}

\newcommand{\Z}{\mathbb{Z}}

\newcommand{\cT}{\mathcal{T}}
\newcommand{\T}{\mathcal{T}}

\newcommand{\cC}{\mathcal{C}}

\newcommand{\cS}{\mathcal{S}}
\newcommand{\It}{I_{\T}}

\def\mult{\operatorname{mult}}

\newcommand{\im}{\mathrm{im}}
\def\str{\operatorname{str}}
\def\spn{\operatorname{span}}

\title{Staged tree models with toric structure}

\author{Christiane G\"orgen, Aida Maraj, and Lisa Nicklasson\\
\small{Max Planck Institute for Mathematics in the Sciences}}
\date{}

\begin{document}

\maketitle

\begin{abstract}
\noindent A staged tree model is a discrete statistical model encoding relationships between events. These models are realised by directed trees with coloured vertices. In algebro-geometric terms, the model consists of points inside a toric variety. For certain trees, called balanced, the model is in fact the intersection of the toric variety and the probability simplex. This gives the model a straightforward description, and has computational advantages. In this paper we show that the class of staged tree models with a toric structure extends far outside of the balanced case, if we allow a change of coordinates. It is an open problem whether all staged tree models have toric structure. 
\end{abstract}

\section{Introduction}


\emph{Staged tree models} are discrete statistical models introduced by Smith and Anderson in 2008~\cite{Smith.Anderson.2008} that record conditional independence relationships among certain events. They are realisable as rooted trees with coloured vertices and labelled edges directed away from the root, called \emph{staged trees}. Vertices in a staged tree represent events, edge labels represent conditional probabilities, and the colours on the vertices represent an equivalence relation---vertices of the same colour have the same outgoing edge labels. 
We use $\theta(e)$ to denote the label associated to an edge $e$. The sum of the labels of all edges emanating from the same vertex in a staged tree must be equal to one. 
The staged tree model itself is then defined as the set of points in $\R^n$ parametrised by multiplying edge labels along the root-to-leaf paths $\lambda_1, \ldots, \lambda_n$ in the staged tree $\T$: 

\begin{equation}\label{eq:parametrisation}
\mathcal{M}_{\T}=\left\lbrace\biggl(\prod_{e\in E(\lambda_r)}\theta(e) \biggr)_{r=1, \ldots, n} \ \left| \  0<\theta(e)<1, \, \sum_{e\in E(v)}\theta(e)=1 \text{ for all vertices } v \text{ in } \T \right.\right\rbrace
\end{equation}
where $E(v)$ is the set of edges which emanate from a vertex $v$, and $E(\lambda_r)$ is the set of edges of the $r^\text{th}$ root-to-leaf path, $r=1,\ldots,n$. The local sum-to-one conditions ensure that the multiplication rule of edge labels along root-to-leaf paths in $\T$ induces a well defined probability distribution.
The visual presentation via a coloured tree has been the key to the staged tree models' increasing popularity in both algebra \cite{DG,AD,duarte2021discrete,Duarte.Solus.2021} and statistics \cite{Keeble.etal.2017,Goergen.etal.2020,rpackage,Leonelli.Varando.2021,Genewein.etal.2021}. 

By the description \cref{eq:parametrisation}, a staged tree model can be interpreted as the variety of the kernel of a map $\varphi_\T$, formally introduced in \cref{eq:varphi}, intersected with the probability simplex. We say that a staged tree model has \emph{toric structure} if the variety of $\kerf$ is a toric variety. This is equivalent to  $\kerf$ being a binomial ideal, possibly after an appropriate linear change of variables. Duarte and G\"orgen~\cite{DG} show that when the tree is balanced, one can safely ignore the sum-to-one conditions, and the ideal $\kerf$ is binomial. This class has wide intersections with Bayesian networks/directed graphical models and hierarchical models. 
The majority of staged trees, however, are not balanced. This is true even for the small portion of staged trees which are also Bayesian networks, making the algebraic study of these models a new and necessary task.

Non-balanced staged trees appear naturally in applications. The most basic example is one of flipping a biased coin with probability $\theta_1$ for heads, and $\theta_2=1-\theta_1$ for tails. If it shows heads, flip it again. The graph in \cref{fig:biasedcoin} is a staged tree presentation of this experiment.  
\begin{figure}
\centering
\includegraphics{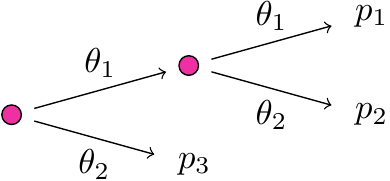}
\caption{The coin-flip model  $\mathcal{M}_\T=\{(\theta_1^2,\theta_1\theta_2, \theta_2)\mid 0<\theta_1,\theta_2<1,\, \theta_1+\theta_2=1\}\subset (0,1)^3$ represented by a binary, non-squarefree, not balanced staged tree. Both internal vertices are in the same stage.}
\label{fig:biasedcoin}
\end{figure}
The points in the corresponding model have the parametrisation $(\theta_1^2,\theta_1\theta_2, \theta_2)$ for positive $\theta_1,\theta_2$ that sum up to one.  They are the solution set inside the probability simplex to the non-binomial equation  $p_1p_3-p_1p_2+p_2^2=0$ in the polynomial ring $\R[p_1,p_2,p_3]$.
 Contrary to the balanced case, as the example suggests, the ideals $\kerf$ of  non-balanced staged trees are not generated by binomials; their generating sets are often complicated. 


The present paper takes a large leap and proves that in many cases non-balanced staged tree models have also a toric structure; one needs to find an appropriate linear change of variables to reveal this friendly structure. For instance, the linear change of coordinates $p_1=q_1, p_2=q_2-q_1$, and $p_3= q_3-q_2$  in the example above bijectively transforms the defining variety into the set of points $(q_1,q_2,q_3)$  that satisfy the binomial equation $q_1q_3-q_2^2=0$.

Why is toric a desirable property? For one, the literature on the algebra of toric varieties is already very rich and with close connections to polyhedral geometry and combinatorics.  Most importantly, toric ideals have shown to be particularly useful in algebraic statistics since with the very first works \cite{diaconis1998algebraic,Pistone.etal.2001,Geiger.etal.2006} in this area.  For instance, generating sets of toric varieties produce Markov bases and contribute in hypothesis testing algorithms~\cite{diaconis1998algebraic}, toric varieties are intrinsically linked to smoothness criteria in exponential families~\cite{Geiger.etal.2001}, and the polytope associated to a toric model is useful when studying the existence of  maximum likelihood estimates~\cite{fienberg2012maximum}. 
 
In previous work, the parametrisation defining a staged tree model  \cref{eq:parametrisation} is often assumed to be squarefree. In other words, in the tree graph no two vertices are in the same stage along the same root-to-leaf path. The theory on staged tree models with a non-squarefree parametrisation is underdeveloped because they pose interpretational ambiguities in a statistical context~\cite{Collazo.etal.2018} and computational challenges in algebraic algorithms~\cite{Goergen.etal.2018}.  However, as the example of the coin flip confirms, non-squarefree staged tree models are very natural, and  both balanced and non-balanced staged trees can be non-squarefree. None of the results in our paper are restricted to squarefree staged tree models, which makes this paper the first to tackle equations defining non-squarefree staged tree models.
 
This article is organised as follows. In \cref{sec:newsec2} we review the background literature on staged trees and introduce combinatorial and algebraic tools for studying these models. We discuss operations called swap and resize on a staged tree, which are used in \cref{sec:balanced} to prove that all balanced staged tree models have a quadratic Gr\"obner basis. It turns out that balanced staged trees carry properties that are central in commutative algebra, namely they are Koszul, normal, and Cohen-Macaulay. In \cref{sec:ideals} we prepare for our investigation of non-balanced trees with toric structure. Informally, the idea is to look for an ideal inside $\kerf$ that is binomial after a linear change of variables, which can be used to detect a binomial structure of $\kerf$. This theory is then set to practice in \cref{sec:extend} where we successfully extend the class of staged trees with toric structure. Our new class contains trees with a certain subtree-inclusion property, balanced trees, and combinations of both. In \cref{sec:one_stage} we explore the structure of staged tree models where all vertices share the same colour. We connect these so-called one-stage trees to Veronese embeddings and show that all binary one-stage trees are toric. The paper ends in \cref{sec:discussion} with a discussion on further research directions, conjectures, and examples of staged trees for which we cannot currently decide whether or not they have toric structure.

\section{Staged tree models}\label{sec:newsec2}
A discrete statistical model can be regarded as a subset of the open probability simplex 
\[{\Delta^\circ_{n-1}}={\{p\in\R^n\mid p_1+\ldots+p_n=1 \text{ and } p_r\in(0,1) \text{ for all } r=1,\ldots,n\}}\]
 of dimension $n-1$. Thus, given a staged tree $\T$, the staged tree model $\mathcal{M}_\T$ in \Cref{eq:parametrisation} defines a discrete statistical model. 
This section starts by studying the  combinatorics of a staged tree, and then introduces the main algebraic object associated to  $\mathcal{M}_\T$.

\subsection{Combinatorics of staged tree models}\label{sec:staged_trees}


Given a staged tree $\T$ let $V$ denote its vertex set and $E$ its edge set. The staged tree $\T$ comes with an equivalence relation  $\sim$ on its vertices. Formally, we say that two inner vertices $u$ and $v$ in $V$ are in the same stage, or share the same colour $c$, if the labels of the outgoing edges of both vertices are the same. Leaves are by convention assumed to be trivially in the same stage since their outgoing edge sets are empty. For two vertices of different stages, the sets of labels of their outgoing edges are assumed to be disjoint. 
In depictions of staged trees we always assume that the outgoing edges of vertices with the same colour are ordered in the same way, such that their labels are pairwise identified from top to bottom. \Cref{fig:biasedcoin,fig:stagedtree} depict first examples. 

For simplicity, leaves of a staged tree will be numbered $1,\dots, n$ and the root-to-leaf paths denoted $\lambda_1,\dots,\lambda_n$. To each leaf an atomic probability $p_r=\prod_{e\in E(\lambda_r)}\theta(e)$ is associated, $r=1,\dots, n$. The notation $p_{[v]}$ is shorthand for the probability of passing through the vertex $v$. Formally, $p_{[v]}$ is the sum of the atomic probabilities $p_r$ for which the root-to-leaf path $\lambda_r$ passes through vertex $v$. 
If $v_i$ is the $i^\text{th}$ child of $v$, the label of the edge $(v,v_i)$ is equal to the fraction $\nicefrac{p_{[v_i]}}{p_{[v]}}$~\cite{Goergen.etal.2020}.
As a consequence, the staged tree model $\mathcal{M}_\T$ can be described via equations $p_{[u]}p_{[v_i]}=p_{[u_i]}p_{[v]} $, for all vertices $u \sim v$. This characterisation motivates the definition of the ideal of model invariants studied in \cref{sec:ideals}.  

From now on, the edge labels will be considered as variables of a polynomial ring $\R[\Theta]$. For a vertex $v\in V$ we  denote by $\T(v)$ the induced subtree of $\T$ which is rooted at $v$ and whose root-to-leaf paths are $v$-to-leaf paths in $\T$. To this subtree we associate a polynomial $t(v)\in\R[\Theta]$ which is equal to  the sum of all products of labels along all $v$-to-leaf paths in $\T(v)$ \cite{Goergen.etal.2018}. This is a key object in our study of balanced staged trees. In particular, we call a staged tree \emph{balanced} \cite{DG} if for any two vertices $u,v$ in the same stage the following is true:
\begin{equation}\tag{$\star$}\label{eq:bal}
t(u_i)t(v_j)=t(u_j)t(v_i) \text{ for all pairs of children of } u \text{ and }v \text{  numbered }  i,j=1,\ldots,k.
\end{equation}
\Cref{fig:stagedtree} shows examples of balanced staged trees. This class is fairly large as all decomposable directed graphical models can be represented by balanced staged trees \cite{DG}. 

Before we can extend the study of balanced staged trees in \cref{sec:balanced}, we need to understand how this property of a tree affects the model it represents.
To this end, note that a staged tree model most often is not uniquely represented by a single staged tree: there may be different parametrisations for the same model, giving rise to different polynomial equations cutting out the same subset of the probability simplex. The class of all staged trees representing the same model is known as the \emph{statistical equivalence class}. Members of such a class are connected to each other by two simple graphical operations: resizes and swaps~\cite{Goergen.Smith.2018}. For the purposes of this paper we present versions of these in a constructive manner.

\begin{figure}
\centering
\begin{subfigure}[b]{0.37\textwidth}
\includegraphics{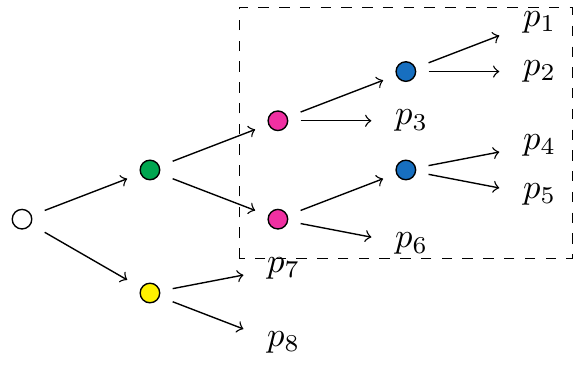}
\caption{A staged tree with a subgraph on which we perform operations.\label{fig:firsttree}}
\end{subfigure}
\quad
\begin{subfigure}[b]{0.28\textwidth}
~\includegraphics{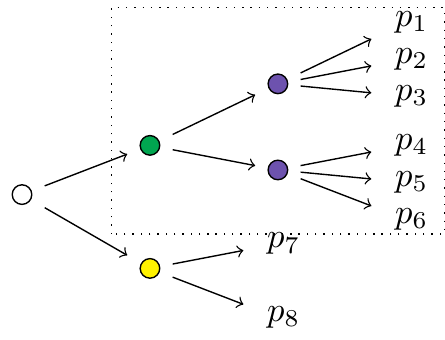}
\caption{The staged tree after performing a resize.\label{fig:resize}}
\end{subfigure}
\quad
\begin{subfigure}[b]{0.28\textwidth}
\includegraphics{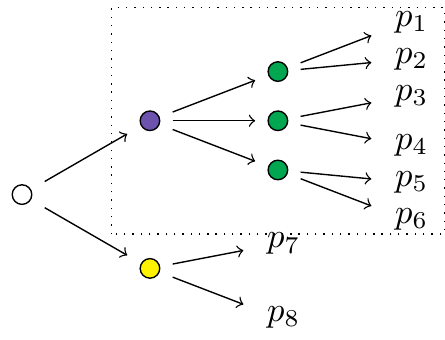}
\caption{The resized staged tree after performing a swap.\label{fig:swap}}
\end{subfigure}
\caption{Three balanced staged trees which are statistically equivalent using a swap and a resize on the highlighted subtrees.\label{fig:stagedtree}}
\end{figure} 

The \emph{swap} operation is illustrated in \cref{fig:resize,fig:swap} where we swap a two-level subtree. Let $\theta_1,\ldots, \theta_k$ be the labels associated to a stage of colour~$c$. Suppose a vertex $v$ is not of the colour $c$, but every term in the subtree polynomial $t(v)$ is divisible by one of the ${\theta_1, \ldots, \theta_k}$. Then we can replace the subtree $\T(v)$ by another tree with the same atomic probabilities but with a root of colour~$c$. 
This is because in this setting every root-to-leaf path in $\T(v)$ goes through a vertex of colour $c$ and we can find a another subtree $\T'(v)\subseteq\T(v)$ with root $v$ and leaves $v_1, \ldots, v_m$ which have the colour~$c$ in $\T$. Let $\T_{v_i\theta_j}$ denote the induced subtree coming after the edge labelled $\theta_j$ of the vertex $u_i$. We obtain a different representation of $\T(v)$ like this. Start with a root of colour $c$ and to each edge glue a copy of $\T'(v)$. Then to each leaf glue the corresponding tree $\T_{v_i\theta_j}$. The map which replaces the original tree by the tree obtained after performing this copying and reglueing is called a swap. 

The \emph{resize} operation is illustrated in \cref{fig:firsttree,fig:resize} where we resize two two-level subtrees. Suppose in a staged tree we have a vertex $u$ such that for every vertex $v$ in the same stage, their children are also pairwise in the same stages, $u_i \sim v_i$ for all $i=1,\ldots,k$. Then for each such $u,v$ we remove their non-leaf children and draw edges from the vertices directly to their respective grandchildren. This introduces a new set of labels for the new edges related to the old ones as follows. Let $w$ be a vertex in the new tree, and $w'$ the corresponding vertex in the original tree. Then if in $t(w)$ we substitute each new label by the product of the two original that it replaced, we get $t(w')$. 
For the resized tree to represent the same model as the original tree we require in addition that no $u_i$ and $u_j$ are in the same stage for $i\not=j$ and these colours do not appear anywhere else in the tree. A resize not respecting these extra conditions were called `na\"ive' in the original presentation  \cite{Goergen.Smith.2018}. But even in the `na\"ive' cases where we do lose information and effectively create a bigger statistical model, these operations are useful tools for the analysis conducted in \cref{sec:extend}. In particular we have the following lemma. 

\begin{lem}\label{lem:resize}
The class of balanced staged trees is invariant under the swap and resize operations.
\end{lem}

\begin{proof}
As for the swap, simply observe that this operation does not change the parametrisation of a staged tree but only locally changes the order of vertices in a subtree $\T(v)$. Hence, $t(v)$ is invariant under the swap, not affecting validity of \cref{eq:bal}.

For the resize, let two vertices $u \sim  v$ be in the same stage in the new tree, after the resize operation has been applied. We want to prove that the condition \cref{eq:bal} is fulfilled, so that $t(u_i)t(v_j)-t(u_j)t(v_i)=0$ for all children of $u$ and $v$, indexed by $i, j=1,\ldots,k$. What we do know is that $ t(u_i)t(v_j)-t(u_j)t(v_i)$ is zero after we substitute a product of labels $\theta_{k \ell}=\tau_k\sigma_\ell$ according to the resize operation, since the original tree was balanced. In other words, this binomial is an element of the ideal generated by $\theta_{k \ell}-\tau_k\sigma_\ell$, namely
\begin{equation}
t(u_i)t(v_j)-t(u_j)t(v_i) \in \langle \theta_{k \ell}-\tau_k\sigma_\ell \rangle_{k, \ell} \subset \R[\Theta]
\end{equation}
where $\Theta$ here denotes the set of labels from both the old tree and the new. But no $\tau_\cdot$'s occur in the new tree, so we must have $t(u_i)t(v_j)-t(u_j)t(v_i)=0$. This proves the claim.
\end{proof}

This result will provide the basis for proving \cref{prop:deg_one_rel}.

\subsection{Algebra of staged tree models}\label{sec:toric_ideals_algebras}


We henceforth treat the atomic probabilities $p_1,\dots,p_n$ of a staged tree $\T$ as variables spanning a polynomial ring denoted $\R[p]$. We will use the short notation $\sumt$ for the ideal of all local  sum-to-one conditions $\langle \sum_{e\in E(v)}\theta(e)-1 \mid v\in V\text{ a non-leaf}\rangle$ in the polynomial ring $\R[\Theta]$. We choose $\R$ as our base field in this paper in the style of real algebraic geometry and algebraic statistics, though all results hold for an arbitrary base field of characteristic zero. Associating the variables of $\R[p]$ to the atomic probabilities induces a ring map 
\begin{equation}\label{eq:varphi}
\varphi_{\T}: \R[p] \to\R[\Theta]/\sumt , \quad p_r \mapsto \prod_{e\in E(\lambda_r)}\theta(e) \text{ for } r=1,\ldots,n.
\end{equation}

We will now use the notation $p_{[v]}$ for the sum of the variables $p_r$ in the ring $\R[p]$ for which the corresponding root-to-leaf path $\lambda_r$ passes through the vertex $v$.
Due to the sum-to-one conditions, the image $\varphi_\T(p_{[v]})$ of an atomic sum $p_{[v]}$ is equal to the product of the labels along the corresponding root-to-$v$ path.

The kernel of $\varphi_\T$ is the \emph{prime ideal of the staged tree} $\T$. 
The variety of this ideal gives an implicit description to the staged tree model $\mathcal{M}_\T$.  
In order to use classical results of algebraic geometry more comfortably, we will mostly work with the homogenised version of $\varphi_\T$,
\begin{equation}\label{eq:homphi}
   {\varphi_{\T}}_{\text{hom}}: \R[p] \rightarrow \R[\Theta,z]\slash \langle\theta-z\rangle, \quad p_{r} \mapsto z^{d-d(\lambda_r)}\prod_{e\in E(\lambda_r)}\theta(e)
\end{equation}
where, in analogy to the notation used in \cref{eq:varphi}, $\langle\theta-z\rangle$ denotes the ideal of homogenised local sum-to-one conditions.
Here $z$ is an artificial parameter, $d(\lambda_r)$ is the number of edges in the root-to-leaf path $\lambda_r$, and  $d$ is the number of edges in the longest  root-to-leaf path in the tree. The integer $d$ will often be regarded as the \emph{depth} of the tree. Each of the variables $p_r$ are mapped to a monomial of degree $d$ under  ${\varphi_{\T}}_{\text{hom}}$.

\begin{rmk}\label{rmk:false_leaves}
The homogenisation process is equivalent to completing the staged tree $\T$ into a staged tree where all leaves are in the same distance $d$ from the root. This is achieved by adding when needed internal vertices of outdegree one and label $z$: compare \cref{fig:bal_col_ex}. Such changes do not affect the model. Hence, we will in notation not distinguish between $\varphi_\T$ and ${\varphi_{\T}}_{\text{hom}}$.
\end{rmk}

An element of significance to this paper is the image of $\varphi_\T$ which is a subring of $\R[\Theta,z]\slash \langle\theta-z\rangle$ denoted $\imf$. We will also refer to $\imf$ as an algebra, as it can be considered as a vector space over $\R$ and a ring at the same time. One can compare the algebras of two staged trees  $\T'\subseteq \T$  sharing the same root. The inclusion $\T'\subseteq \T$ here indicates that  $\T'$ can be obtained  by removing induced subtrees $\T(v)$ for vertices $v$ in $\T$.
 \begin{lem}
\label{lm:inclusion} 
 For any two staged trees $\T'\subseteq \T$ that share the same root one has  $\im \, \varphi_{\T'}\subseteq \imf$. 
\end{lem}

\begin{proof}
The image of $\varphi_{\T'}$ is generated by the products of labels of root-to-leaf paths in $\T'$. Each of these paths is a root-to-$v$ path in $\T$, for some vertex $v$. Hence the same product lies in the image of $\varphi_\T$, as it is the image of $p_{[v]}$. 
\end{proof}

The central objective of this paper is to provide conditions under which the ideal $\kerf$ is toric, possibly after a linear change of variables. Recall thus that a prime ideal in $\R[p]$ is called toric if it is generated by binomials, or equivalently is the kernel of a monomial map $\varphi$ from $\R[p]$ to a Laurent ring $\R[y_1^{\pm 1}, \ldots, y_m^{\pm 1}]$. That $\varphi$ is a monomial map means that each variable $p_r$ is mapped to a monomial. The kernel of such a map is generated by differences of monomials $p^a-p^b$ such that $\varphi(p^a)=\varphi(p^b)$. The image of $\varphi$ is the subalgebra of $\R[y_1^{\pm 1}, \ldots, y_m^{\pm 1}]$ generated by the monomials $\varphi(p_1), \ldots, \varphi(p_n)$. We say that an algebra is a \emph{monomial algebra} if it is generated by monomials in a Laurent ring or a polynomial ring. 

Even though the map \cref{eq:varphi} appears monomial, its target ring is a quotient ring instead of a Laurent ring, and we cannot conclude that $\kerf$ defines a toric variety; recall here the coin flip example from the introduction. For staged tree models, the ideal $\ker \varphi_{\T}$ is toric in the variables $p_1, \ldots, p_n$ if and only if the tree $\T$ is balanced.
\begin{thm}{\cite[Theorem~10]{DG}}\label{thm:balanced}
The ideal $\kerf$ is a toric ideal if and only if $\T$ is a balanced staged tree. 
\end{thm}
Hence, if a non-balanced tree has a toric structure, the structure must appear only after a linear change of coordinates. This can be done either by explicitly finding  a linear change of variables which makes $\kerf$ a binomial ideal, or to prove that $\imf$ is a monomial algebra as we state below.
\begin{prop}\label{prop:monomial_algebra}
Suppose the algebra $\imf$ is isomorphic to a monomial algebra. Then the ideal $\ker\varphi_\T$ is toric after an appropriate linear change of coordinates. 
\end{prop}

\begin{proof}
Let $\imf$ be isomorphic to an algebra minimally generated by monomials  $m_1,\ldots,m_r$ in some polynomial ring via a map $f$. Then we have a composition of maps
\begin{equation}
\R[p] \overset{\varphi_\T}{\longrightarrow} \R[\Theta,z]/\langle \theta-z \rangle \overset{f}{\longrightarrow} \R[m_1, \ldots, m_r]
\end{equation}
which is a surjective homomorphism. We claim that the linear span of $f \circ \varphi_\T (\{p_1, \ldots, p_n\})$ equals the linear span of $m_1, \ldots, m_r$. Every monomial $m_i$ is in the image of $f \circ \varphi_\T$, and if it would not be the image of a linear form this would contradict the fact that $m_1, \ldots, m_r$ are minimal generators. Then every $p_i$ must be mapped to a linear combination of $m_1, \ldots, m_r$, since otherwise we would get a non-homogeneous binomial generator of $\ker \varphi_\T$. It follows that there are $n$ linearly independent linear forms $\ell_1,\ldots,\ell_n$ in the union of preimages $\bigcup_{i=1}^k\varphi_\T^{-1}\circ f^{-1}(m_i)$. The ideal $\kerf$ is binomial after the linear change of variables given by $\ell_1,\ldots,\ell_n$. 
\end{proof}

\begin{ex}\label{ex:algebra_binary} 
Let $\T$ be the tree in \cref{fig:monomial_alg_ex}, and let $\Theta^4$ denote the set of monomials of degree four in the variables $\theta_1$, $\theta_2$. Then $\imf$ is the subalgebra of $\R[\theta_1,\theta_2,z]/\langle \theta_1+\theta_2-z\rangle$ generated by $\theta_1^2z^2$, $\theta_1\theta_2z^2$, $\theta_1^2\theta_2z$, $\theta_1^1\theta_2^2$, $\theta_1\theta_2^3$, and $\theta_2^2z^2$.  All of the generators can be expressed as linear combinations of the monomials in $\Theta^4$, using the relation $z=\theta_1+\theta_2$. Conversely, one can verify that the five monomials in $\Theta^4$ are the images of the linear forms $\ell_1=p_1-2p_3+p_4$, $\ell_2=p_3-p_4$, $\ell_3=p_4$, $\ell_4=p_5$,  $\ell_5=p_6-2p_5-p_4$ under $\varphi_\T$. For instance
\begin{equation}
\varphi_\T(p_1-2p_3+p_4)=\theta_1^2z^2-2\theta_1^2\theta_2z+\theta_1^2\theta_2^2=\theta_1^4+2\theta_1^3\theta_2+\theta_1^2\theta_2^2-2\theta_1^3\theta_2-2\theta_1^2\theta_2^2+\theta_1^2\theta_2^2=\theta_1^4.
\end{equation}
It follows that $\Theta^4$ is a different generating set for the algebra $\imf$. As $\R[\Theta^4] \cap \langle \theta_1+\theta_2-z \rangle = \{0\}$, the image of $\varphi_\T$ is isomorphic to the monomial algebra $\R[\Theta^4]$. To get the change of variables which makes $\ker \varphi_\T$ toric we choose another linear form $\ell_6$, not in the span of $\ell_1, \ldots, \ell_5$, which maps to a monomial of $\Theta^4$. For example we may take $\ell_6=p_2-p_3-p_5$, as this maps to $\theta_1^2\theta_2^2$. The ideal $\ker \varphi_\T$ is then the kernel of the monomial map $\R[\ell_1, \ldots, \ell_6] \to \R[\Theta^4]$.  
\end{ex}

\begin{figure}
\centering
\includegraphics{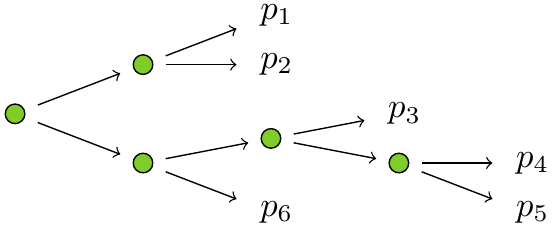}
\caption{A staged tree with edge labels $\theta_1$ and $\theta_2$ that has toric structure.}
\label{fig:monomial_alg_ex}
\end{figure}

A first application of \cref{prop:monomial_algebra} is that if two distinct trees share the same algebra then they must both be toric or non-toric. The same is true if the two algebras change by a permutation or when the parameters of a staged tree are permutations of parameters of another other staged tree.

\begin{cor}\label{lm:permuting}
If for staged trees $\T$ and $\T'$ the algebras  $\imf$ and $\im\,\varphi_{\T'}$  are equal or they change only by a permutation of variables, the staged tree models $\T$ and $\T'$ are simultaneously toric or non-toric. 
\end{cor}

\section{Balanced staged tree models}
\label{sec:balanced}


As stated in Theorem \ref{thm:balanced} the ideal $\kerf$ is toric when the tree $\T$ is balanced. The monomial map defining this toric ideal is precisely the map (\ref{eq:homphi}) with the quotient ring $\R[\Theta,z]/\langle \theta-z\rangle$ replaced by the polynomial ring $\R[\Theta,z]$.
In this section we continue the study of balanced staged trees and their prime ideals. 

\subsection{The combinatorial structure of balanced trees}
Our main result on the combinatorial structure of balanced trees is \cref{thm:colour}, which states that a balanced staged tree model always can be represented by a tree with a certain colour structure. To obtain this colour structure we apply the homogenisation process already in the tree, as described in Remark~\ref{rmk:false_leaves}. We also allow the out-degree one vertices to appear anywhere in the tree, not just as extensions in the end of a branch. See \cref{fig:bal_col_ex} for an example. 

\begin{figure}
\centering
\includegraphics{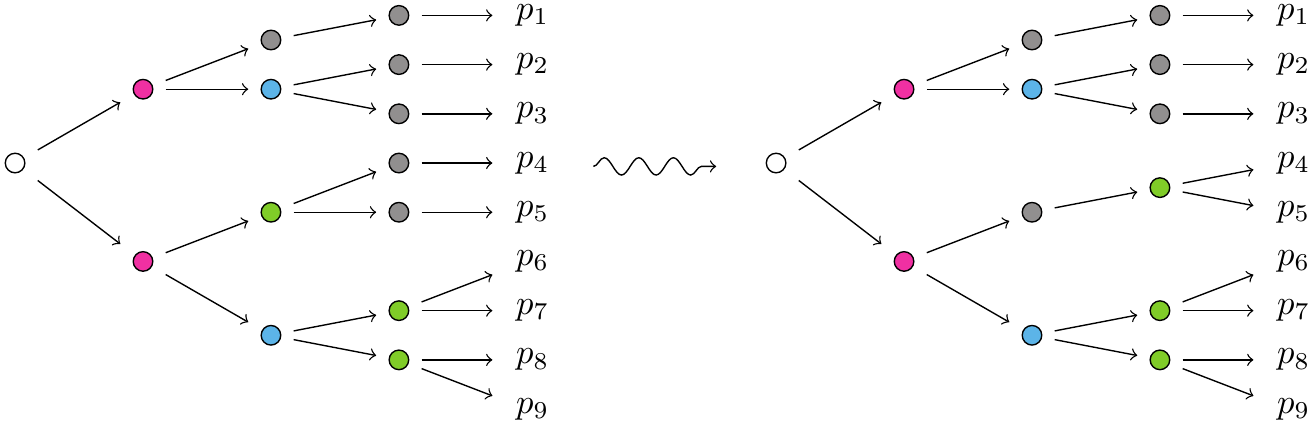}
\caption{Colour structure in a balanced tree.}
\label{fig:bal_col_ex}
\end{figure}

\begin{thm}\label{thm:colour}
To any balanced staged tree we can apply the swap operation so that for every pair of vertices $u$ and $v$ in the same stage, one of the following holds. 
\begin{enumerate}
\item For each index $i$ the two children $u_i$ and $v_i$ are in the same stage. 
\item The children of $u$ are all in the same stage. The same holds for $v$.   
\end{enumerate} 
\end{thm}

In order to prove Theorem \ref{thm:colour} we need some preparations. 
We define the \emph{multiplicity} of a colour $c$ at a vertex $v$, denoted $\mult_c(v)$, to be the greatest number $m$ for which every $v$-to-leaf path goes through at least $m$ vertices of the colour $c$. 
If $\theta_1, \ldots, \theta_k$ are the labels associated to the colour $c$ we can also say that $\mult_c(v)$ is the greatest number $m$ for which every term in $t(v)$ is divisible by a monomial in $\theta_1, \ldots, \theta_k$ of degree $m$. If $\mult_c(v)>0$ we can use the swap operation to give the induced subtree $\T(v)$ a root of colour $c$. Notice also that applying the swap operation to $\T(v)$, or to an induced subtree inside $\T(v)$, does not change the multiplicities.   

\begin{lem}\label{lem:colour}
Let $u$ and $v$ be vertices of the same stage in a balanced staged tree, and assume $\mult_c(u_i)>\mult_c(u_j)$ for some colour $c$ and some indices $i,j$. Then $\mult_c(v_i)>\mult_c(v_j)$. 
\end{lem}

\begin{proof}
The equality $t(u_i)t(v_j)=t(u_j)t(v_i)$ implies $\mult_c(u_i)+\mult_c(v_j)=\mult_c(u_j)+\mult_c(v_i)$.
If $\mult_c(u_i)>\mult_c(u_j)$ then we must have $\mult_c(v_i)>\mult_c(v_j)$ for the above equality to hold. 
\end{proof}

\begin{proof}[Proof of Theorem \ref{thm:colour}]
In this proof vertices $v$ such that 
\begin{equation}\label{eq:proof_colourthm}
\mbox{for each} \ i \ \mbox{there is a colour} \ c_i \ \mbox{and a} \ j \ \mbox{such that} \ \mult_{c_i}(v_i)>\mult_{c_i}(v_j) 
\end{equation}
will play an important role. If $v$ has the property (\ref{eq:proof_colourthm}) then we can give each $v_i$ the colour $c_i$, using the swap operation on $\T(v_i)$. If $v$ does not have the property (\ref{eq:proof_colourthm}), then either all its children are leaves, or we can use the swap operation on the $\T(v_i)$'s to give all children of $v$ the same colour. It follows from Lemma \ref{lem:colour} that if $u \sim v$ then $u$ satisfies (\ref{eq:proof_colourthm}) if and only if $v$ does with the same colours.

We prove the theorem by providing an algorithm which will give the desired colour structure, assuming the tree is balanced. The algorithm goes through the internal vertices of out-degree greater than one whose children are not leaves. We visit the vertices in weakly increasing order, according to the distance from the root, and do the following.  

\begin{itemize}
\item If $v$ has a colour we have not seen before, we check if $v$ satisfies (\ref{eq:proof_colourthm}). If it does, we give each $v_i$ the colour $c_i$, by using the swap operation. If not, we give all $v_i$'s the same colour. 
\item If $v$ has a colour that we have seen before, we look at the children of the previous vertex $u$ of that colour. If the children of $u$ does not all have the same colour, it means the $u$ has the property (\ref{eq:proof_colourthm}). Then so does $v$, and we can give each $v_i$ the same colour as $u_i$. If all the children of $u$ have the same colour, it means that $u$ did not satisfy (\ref{eq:proof_colourthm}). Then neither does $v$, and we give all children of $v$ the same colour.  
\end{itemize}%
Notice that these steps do not change the colour of any vertex we have already visited, or their children. Therefore, condition \textit{1} or \textit{2} in Theorem \ref{thm:colour} will hold for all vertices $u$ and $v$ in the same stage of out-degree greater than one. For the vertices of out-degree one condition \textit{2} always holds, so we are done.    
\end{proof}

In squarefree staged trees, the root-to-leaf path $\lambda_r$ can be read off uniquely from the monomial  $\varphi_\T(p_r)$. In other words, the ideal $\kerf$ will not contain any relations of degree one. Next, we will study the structure of balanced trees which are \emph{not} squarefree. We use the notation $v_{ij}$ for the grandchildren of $v$, meaning that $v_{ij}$ is the $j$-th child of $v_i$.  

\begin{lem}\label{lem:deg_one}
In a balanced tree, suppose we have a vertex $v$ which is in the same stage as all of its children. Then $t(v_{ij})=t(v_{ji})$.
\end{lem}

\begin{proof}
Since $v \sim v_i$ we have $t(v_i)t(v_{ij})=t(v_j)t(v_{ii})$ which can also be written as
\begin{equation}\label{eq:lem_deg_one}
(\theta_1t(v_{i1}) + \theta_2t(v_{i2}) +\dots +  \theta_kt(v_{ik}))t(v_{ij}) = (\theta_1t(v_{j1}) + \theta_2t(v_{j2}) +\dots +  \theta_kt(v_{jk}))t(v_{ii}) .
\end{equation}
As $v_i \sim v_j$ we also have $t(v_{ii})t(v_{jm})=t(v_{im})t(v_{ji})$. Applying this to the right hand side of (\ref{eq:lem_deg_one}) we get
\begin{equation}
(\theta_1t(v_{i1}) + \theta_2t(v_{i2}) +\dots +  \theta_kt(v_{ik}))t(v_{ij}) =(\theta_1t(v_{i1}) + \theta_2t(v_{i2}) +\dots +  \theta_kt(v_{ik}))t(v_{ji})
\end{equation}
and it follows that $t(v_{ij})=t(v_{ji})$.
\end{proof}

\begin{prop}\label{prop:deg_one_rel}
Let $u$ and $v$ be vertices of the same stage in a balanced staged tree. Suppose there is a path $\lambda_r$ that goes through both $u$ and $v_j$, for some $j$, where $u$ comes first. Then there is a path $\lambda_s$ going through $u_j$ such that $\varphi_\T(p_r)=\varphi_\T(p_s)$.   
\end{prop}
\begin{proof}
It is enough to consider the case when $u$ is the root, as we can restrict to the subtree $\T(u)$. 
As a first step we apply the algorithm in the proof of Theorem \ref{thm:colour}. This does not change the colour of the root, but it might of course change the appearance of the rest of the tree. However, if we prove the statement for all $v \sim u$ in the new tree it will hold for all $v \sim u$ in the old tree as both trees represent the same model in the same parametrisation. 

The proof is by induction over the depth of the tree. The smallest depth where we can have two vertices in the same stage in the same path is two. One can easily check that for such a tree to be balanced we need all internal vertices to be in the same stage. In this case we have $\varphi_\T(p_r)=\theta_i \theta_j$ and $\varphi_\T(p_s)=\theta_j\theta_i$. 

For the induction step we consider three cases. 
\begin{enumerate}
\item \underline{The children of $u$ do not all have the same colour.} In this case we have $u_i \sim v_i$ for all vertices $v$ that are in the same stage as $u$. Then we use the resize operation on all the vertices in the same stage as $u$. The new tree is also balanced by Lemma \ref{lem:resize}. As the depth has decreased, it follows by induction that the statement is true for this tree. We can easily redo the resize to find the path $\lambda_s$ in the original tree.
\item \underline{The children of $u$ all have the same colour, but not the same colour as $u$ itself.} This case is also illustrated in \cref{fig:prop_deg1}. Here we use the swap operation on $u$ and its children. This results in a new tree with a root followed by a number of induced subtrees whose roots are in the same colour as $u$ in the original tree. By induction there is a path $\lambda_s$ in the same subtree as $\lambda_r$, with the desired properties. We can find this path $\lambda_s$ in the original tree by swapping the root and its children one more time.
\item \underline{The children of $u$ have the same colour as $u$.} If we consider $v$ to be one on the children of $u$, the result follows from Lemma \ref{lem:deg_one}. Otherwise we can swap $u$ and its children and continue as in case 2. \qedhere 
\end{enumerate}
\end{proof}
\begin{figure}
\centering
\begin{subfigure}[b]{0.45\textwidth}
\centering
\includegraphics{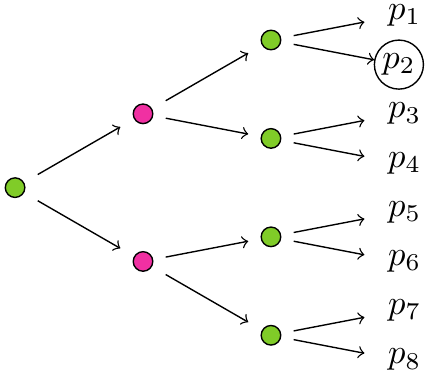}
\caption{Before the swap.\label{fig:A}}
\end{subfigure}
\quad
\begin{subfigure}[b]{0.45\textwidth}
\centering
\includegraphics{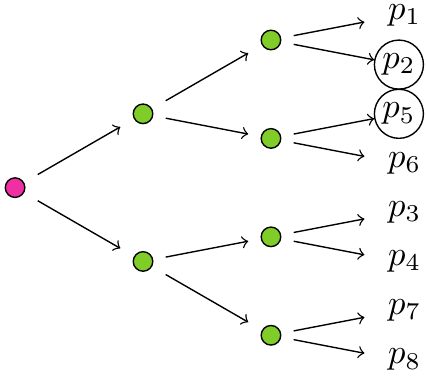}
\caption{After the swap.\label{fig:B}}
\end{subfigure}
\caption{Illustration of case 2 in the proof of Proposition \ref{prop:deg_one_rel}. By applying the swap operation we find the $p_s$ such that $\varphi_\T(p_2)=\varphi_\T(p_s)$ inside a smaller induced subtree. Here $s=5$. }
\label{fig:prop_deg1}
\end{figure}

\subsection{The toric ideal associated to a balanced tree}
We now turn to the problem of finding a generating set for the toric ideal $\kerf$, when $\T$ is balanced. We start with a quick review of the concept of \emph{Gröbner bases} for ideals in polynomials rings. For more details we refer the reader to~\cite{EH}. Every Gröbner basis relies on a monomial order $\succ$. Here we will use the \emph{Degree Reverse Lexicographic} monomial order (DegRevLex). Assume the variables $p_1, \ldots, p_n$ are numbered from top to bottom in the tree, with $p_1$ in the top. We order the variables by $p_1\succ p_2\succ \dots \succ p_n$. More generally, the monomials are ordered by $p_1^{\alpha_1} \cdots p_n^{\alpha_n}\succ p_1^{\beta_1} \cdots p_n^{\beta_n}$ in DegRevLex if $\alpha_1+ \dots + \alpha_n\succ\beta_1 + \dots + \beta_n$ or $\alpha_1+ \dots + \alpha_n=\beta_1 + \dots + \beta_n$ and there is an $i$ for which $\alpha_i<\beta_i$ and $\alpha_j=\beta_j$ for all $j> i$. Every polynomial $f$ is a linear combination of monomials, and the greatest monomial according to the given order provides the \emph{leading term} of $f$. Now, let $I$ be an ideal in a polynomial ring, and let $G$ be a set of polynomials in $I$. The set $G$ is a Gröbner basis for $I$ if the leading term of any polynomial in $I$ is divisible by the leading term of a polynomial in $G$. A fundamental fact is that a Gröbner basis for an ideal $I$ is a generating set for $I$. 

It was conjectured by Duarte and Görgen that the toric ideal of a balanced staged tree have a Gröbner basis of binomials of degree two, \cite[Conjecture 17]{DG}. The conjecture was proved by Ananiadi and Duarte~\cite{AD} in the case of stratified staged trees with all leaves on the same distance from the root. We shall now prove the conjecture in full generality, with the modification that degree one generators are needed in the non-squarefree case. An example of a balanced staged tree which is not squarefree, and hence not stratified, can be found in Figure \ref{fig:prop_deg1}.

Let $\T$ be a balanced staged tree. Suppose $u$ and $v$ are vertices in the same stage in $\T$ and let $\theta_1, \ldots, \theta_k$ be the labels of this stage. Let $m_u$ be the product of the labels along the path from the root to $u$, and similarly for $m_v$. If we multiply (\ref{eq:bal}) by $m_um_v\theta_i \theta_j$ we get 
\begin{equation}
(m_u\theta_it(u_i))(m_v\theta_jt(v_j))=(m_u\theta_jt(u_j))(m_v\theta_it(v_i)).
\end{equation}  
Notice that the terms in each of the factors correspond to a root-to-leaf path. As every term has coefficient 1, there is no cancellation. Hence the above equality gives rise to a number of binomials $p_{r_1}p_{r_2}-p_{s_1}p_{s_2}$ in $\kerf$, where the path $\lambda_{r_1}$ goes through $u_i$, the path $\lambda_{r_2}$ through $v_{j}$, the path $\lambda_{s_1}$ through $u_{j}$, and $\lambda_{s_2}$ through $v_{i}$. Let $G$ be the set of all such binomials, together with all binomials of degree one in $\kerf$. The set $G$ for the tree in \cref{fig:A} is given in Example \ref{ex:GB}. We shall see that $G$ is a Gröbner basis for $\kerf$.

\begin{thm}\label{thm:GB12}
For a balanced staged tree the set $G$ of binomials of degree one and two defined above is a Gröbner basis for the ideal $\kerf$ w.\,r.\,t.\ DegRevLex. 
\end{thm}

\begin{proof}
We know that $\kerf$ is generated by homogeneous binomials. It follows from Buchberger's algorithm~\cite{EH} that every binomial ideal has a Gröbner basis of binomials. Hence we are done if we can prove that the leading term of a binomial in $\kerf$ is divisible by the leading term of a binomial of in $G$. Let $f=p_{r_1} \dots p_{r_d}-p_{s_1} \cdots p_{s_d} \in \kerf$ where $p_{r_1} \cdots p_{r_d}$ is the leading term. We assume that $p_{r_1}\succ \dots \succ p_{r_d}$ and $p_{s_1}\succ \dots \succ p_{s_d}$. We may also assume that $f$ is not divisible by a monomial. This implies that $p_{r_d}\succ p_{s_d}$. 

Let $u$ be the vertex where the two paths $\lambda_{r_d}$ and $\lambda_{s_d}$ split. Say $\lambda_{r_d}$ then follows an edge labelled $\theta_i$, and $\lambda_{s_d}$ an edge labelled $\theta_j$. As $p_{r_d}\succ p_{s_d}$ we have $i<j$. Let $m$ be the product of the labels along the path from the root to $u$. Then
\begin{equation}
\varphi_\T(p_{r_1}) \cdots \varphi_\T(p_{r_{d-1}})\frac{\varphi_\T(p_{r_d})}{m}=\varphi_\T(p_{s_1}) \cdots \varphi_\T(p_{s_{d-1}})\frac{\varphi_\T(p_{s_d})}{m}
\end{equation}
and since the monomial $\varphi_\T(p_{s_d})/m$ is divisible by $\theta_j$, one of the factors in the left hand side must be divisible by $\theta_j$ as well. Notice that none of $\lambda_{r_1}, \ldots, \lambda_{r_d}$ goes through $u_j$, as $p_{r_1}\succ \dots \succ p_{r_d}$ and $\lambda_{r_d}$ takes the $i$-th edge at $u$. Then we have the following cases.
\begin{enumerate}
\item \underline{Some $\lambda_{r_t}$ goes through $u$ and an edge labelled $\theta_j$, and the $\theta_j$-edge appears after $u$.} Then by Proposition \ref{prop:deg_one_rel} there is a $\lambda_{r'_t}$ that takes the $j$-th edge at $u$ and $p_{r_t}-p_{r'_t} \in G$. As $i<j$ we have $p_{r_t}\succeq p_{r_d}\succ p_{r'_t}$, and $p_{r_t}$ is the leading term of $p_{r_t}-p_{r'_t}$. 
\item \underline{Some $\lambda_{r_t}$, with $t<d$, goes through $u$ and an edge labelled $\theta_j$, and the $\theta_j$-edge  appears before $u$.}\\ The $\theta_j$-edge is an outgoing edge of some vertex $v$ in the same stage as $u$. As $v$ and $v_j$ lie on the path leading to $u$, the path $\lambda_{s_d}$ must also go through $v_j$. At $u$ the path $\lambda_{r_t}$ takes the $l$-th edge for some $l<j$. Then we have $p_{r_t}p_{r_d}- p_{r'_d}p_{r'_t} \in G$ such that $p_{r'_t}$ takes the $j$-th edge at $u$, and $p_{r'_d}$ takes the $l$-th edge at $v$. Then $p_{r'_d}\succ p_{r'_t}$ and $p_{r_s}\succ p_{r'_t}$, which makes $p_{r_t}p_{r_d}$ the leading term. 
\item \underline{Some $\lambda_{r_t}$, with $t<d$, goes through an edge labelled $\theta_j$ and does not go through $u$.} So there is a vertex $v$ in the same stage as $u$, and they are on different branches. Since $p_{r_t}\succ p_{r_d}$ the vertex $v$ must sit above $u$ in the tree. We have $p_{r_t}p_{r_d}-p_{r'_t}p_{r'_d} \in G$ such that $\lambda_{r'_d}$ takes the $j$-th edge at $u$ and $\lambda_{r'_t}$ takes the $i$-th edge at $v$. Then $p_{r'_t} \succ p_{r'_d}$ and $p_{r_d}\succ p_{r'_d}$, which makes $p_{r_t}p_{r_d}$ the leading term.  
\end{enumerate}%
In all three cases we have found an element in $G$ with a leading term which divides the leading term of $f$.
\end{proof}

Sometimes it can be useful to identify variables that are mapped to the same monomial under $\varphi_\T$, and in this way get rid of the degree one relations. For this purpose, let $\R[\bar{p}]$ be the polynomial ring on the subset of the variables $p_1, \ldots, p_n$ obtained by removing $p_r$ if there is a $s>r$ such that $\varphi_\T(p_r)=\varphi_\T(p_s)$. Let $\bar{\varphi}_{\T}$ be the map $\varphi_\T$ restricted to $\R[\bar{p}]$. The two maps have identical images.

\begin{thm}\label{thm:GB2}
For a balanced staged tree, the degree two binomials of $\ker \bar \varphi_\T$ are a Gröbner basis for $\ker \bar \varphi_\T$ w.\,r.\,t.\ DegRevLex. 
\end{thm}
\begin{proof}
Let $f=p_{r_1} \!\cdots p_{r_d}-p_{s_1}\! \cdots p_{s_d} \in \ker \bar \varphi$, where  $p_{r_1}\! \cdots p_{r_d}$ is the leading term. As we also have $f \in \kerf$ the leading term of $f$ is divisible by the leading term of a binomial $g \in \kerf$ of degree one or two, by Theorem \ref{thm:GB12}. However, the leading terms of the degree one binomials in $\kerf$ are precisely the variables we removed in the construction of $\R[\bar p]$.  Hence $g$ must be of degree two. The leading term of $g$ is monomial in (one or two of) the variables $p_{r_1}, \ldots, p_{r_d} \in \R[\bar p]$. If the non-leading term of $g$ is divisible by some $p_i$ for which there is a $j> i$ such that $\varphi_\T(p_i)=\varphi_\T(p_j)$, we substitute $p_i$ by $p_j$ in $g$. This produces a binomial $\bar g \in \ker \bar \varphi_\T$ and does not change the leading term. Hence $p_{r_1}\! \cdots p_{r_d}$ is divisible by the leading term of the degree two binomial $\bar g \in \ker \bar \varphi_\T$.
\end{proof}

\begin{ex}\label{ex:GB}
Consider the tree in \cref{fig:A}. The Gröbner basis of degree one and two binomials of $\kerf$ from Theorem \ref{thm:GB12} is
\begin{align}
&p_2-p_5, \ p_4-p_7, \ p_2p_3-p_1p_4, \ p_2p_5-p_1p_6, \ p_2p_7-p_1p_8, \ p_4p_5-p_3p_6, \ p_4p_7-p_3p_8, \ p_6p_7-p_5p_8, \\
&p_2^2-p_1p_6,\  p_2p_3-p_1p_7, \ p_2p_4-p_1p_8,\ p_1p_4-p_3p_5,\ p_2p_4-p_3p_6,\ p_4^2-p_3p_8,\ p_5^2-p_1p_6,\ p_3p_6-p_5p_7, \\
&p_4p_6-p_5p_8,\ p_5p_7-p_1p_8,\ p_6p_7-p_2p_8,\ p_7^2-p_3p_8,\ p_3p_5-p_1p_7,\ p_3p_6-p_1p_8,\ p_4p_5-p_2p_7,\  p_4p_6-p_2p_8.
\end{align}
The Gröbner basis for $\ker \bar \varphi_\T \subset \R[p_1,p_3,p_5,p_6,p_7,p_8]$ in Theorem \ref{thm:GB2} is 
\begin{equation}
p_6p_7-p_5p_8,\  p_5^2-p_1p_6,\ p_3p_6-p_5p_7,\ p_5p_7-p_1p_8,\ p_7^2-p_3p_8,\  p_3p_5-p_1p_7 ,\ p_3p_6-p_1p_8.
\end{equation}
\end{ex}

\subsection{The monomial algebra associated to a balanced tree}

For a balanced tree $\T$ let $A_\T$ denote the subalgebra of $\R[\Theta,z]$ generated by the monomials corresponding to the root-to-leaf paths. As this set of monomials is exactly the parametrisation for the toric ideal $\kerf$ we have $A_\T \cong  \R[p] / \kerf$. In this section we will discuss some properties of $A_\T$ that are central in commutative algebra. In particular we shall see that $A_\T$ is \emph{Koszul}, \emph{normal}, and \emph{Cohen-Macaulay}. We will briefly recall the definitions of these properties. For a more extensive review we refer the reader to \cite{EH}.

A $K$-algebra $A$ is \emph{Koszul} if the field $K$ has a linear free resolution over $A$. We recall two fundamental results about Koszul algebras. Suppose $A=K[x_1, \ldots, x_n]/I$. If $A$ is Koszul then $I$ is generated in degree two. Moreover, if $I$ has a Gröbner basis consisting of polynomials of degree two then $A$ is Koszul~\cite{Anick}. These give us an immediate corollary of Theorem \ref{thm:GB2}. 

\begin{cor}
For a balanced staged tree $\T$, the associated monomial algebra $A_\T$ is Koszul.
\end{cor}

An algebra generated by monomials, such as $A_\T$, can be considered a semigroup ring. The semigroup is the set of monomials in the ring, with multiplication as the semigroup operation. A semigroup $M$ is called \emph{normal} if it is finitely generated and has the following property. If there are $a, b, c \in M$ such that $ab^N=c^N$ for some positive integer $N$ then there is a $d \in M$ such that $a=d^N$. The ring $A_\T$ is a Noetherian domain, meaning that it is a normal ring if it is integrally closed in its field of fractions. Moreover, recall that a ring $R$ is \emph{Cohen-Macaulay} if $\dim(R)=\operatorname{depth}(R)$, where in general $\dim(R)\ge\operatorname{depth}(R)$. We summarise two important results on normal semigroup rings.

\begin{thm}{\cite[Proposition 1, Theorem 1]{Hochster}}
Let $M$ be a semigroup of monomials, and let $K[M]$ denote the semigroup ring over a field $K$. Then $K[M]$ is normal if and only if $M$ is normal. Moreover, if $K[M]$ is normal, then it is Cohen-Macaulay.
\end{thm}

Let $\T$ be a balanced staged tree, and let $M$ be the semigroup generated by the monomials $\varphi_\T(p_r)$, $r=1, \ldots, n$, considered as monomials in $\R[\Theta,z]$. We shall prove that $M$ is normal, and hence $A_\T$ is normal. 

So, suppose we have a monomial $m \in M$ and two monomials $\varphi_\T(p_{r_1}) \cdots \varphi_\T(p_{r_d})$ and $\varphi_\T(p_{s_1}) \cdots \varphi_\T(p_{s_e})$ such that $m(\varphi_\T(p_{r_1}) \cdots \varphi_\T(p_{r_d}))^N=(\varphi_\T(p_{s_1}) \cdots \varphi_\T(p_{s_e}))^N$ for some positive integer $N$. Assume we have indexed $p_{s_1}, \ldots, p_{s_e}$ so that the path $\lambda_{s_1}$ agrees with $\lambda_{r_1}$ in as many steps as possible. Suppose they separate in a vertex $u$, where $\lambda_{r_1}$ takes the edge labelled $\theta_i$ and $\lambda_{s_1}$ the edge labelled $\theta_j$. Then there must be a matching $\theta_j$ in $\varphi_\T(p_{s_1}) \cdots \varphi_\T(p_{s_e})$. Here we can follow the same steps as in the proof of Theorem \ref{thm:GB12} and get that there is $p_{s'_1}$ and $p_{s'_t}$ such that $\varphi_\T(p_{s'_1})\varphi_\T( p_{s'_t})=\varphi_\T(p_{s_1})\varphi_\T(p_{s_t})$, for some $t$, and $\lambda_{s'_1}$ agrees with $\lambda_{r_1}$ in one more step. If $p_{s'_1} \ne p_{r_1}$ we repeat this argument. Let $p_{s'_1}, \ldots, p_{s'_e}$ be the variables we end up with after repeating the process enough times to get $p_{s'_1} = p_{r_1}$.
Now we have
\begin{equation}
m \varphi_\T(p_{r_1})^N(\varphi_\T(p_{r_2}) \cdots \varphi_\T(p_{r_d}))^N=\varphi_\T(p_{r_1})^N(\varphi(p_{s'_2}) \cdots \varphi_\T(p_{s'_e}))^N 
\end{equation}
which implies 
\begin{equation}  
m (\varphi_\T(p_{r_2}) \cdots \varphi_\T(p_{r_d}))^N=(\varphi(p_{s'_2}) \cdots \varphi_\T(p_{s'_e}))^N .
\end{equation}
We can continue this process until we have cancelled all factors except $m$ in the left hand side. Then $m$ is indeed the $N$-th power of an element in $M$. 
To summarise we have proved the following theorem.

\begin{thm}\label{thm:normal}
For a balanced staged tree $\T$, the associated monomial algebra $A_\T$ is normal and Cohen-Macaulay.
\end{thm}

For a non-balanced tree the algebra $\R[p]/\kerf$ need not to be Koszul, normal, or Cohen-Macaulay.

\begin{figure}
\centering
\includegraphics{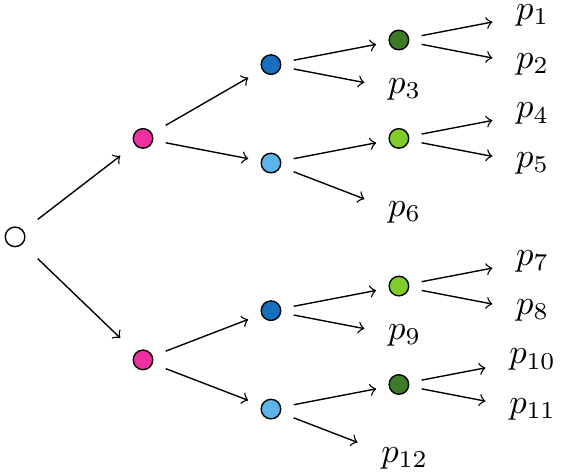}
\caption{A staged tree for which $A_\T$ is not Koszul, normal, or Cohen-Macaulay.}
\label{fig:nonCM}
\end{figure}

\begin{ex}\label{ex:nonCM}
The tree in \cref{fig:nonCM} is not balanced, but the ideal $\kerf$ is toric after a linear change of variables. We will return to this fact in Example~\ref{ex:continued}. A computation in Macaulay2~\cite{M2} shows that the algebra is neither normal nor Cohen-Macaulay. A minimal generating set for $\kerf$ needs binomials of degree two and four, so the quotient ring is not Koszul. This also provides a counter-example to \cite[Conjecture 16]{DG} on the generating set of the prime ideal of a staged tree, as the given generators all are of degree two.
\end{ex}

\section{Ideal of minors for a staged tree model }\label{sec:ideals}


Associated to a staged tree is its \emph{ideal of model invariants}~\cite{DG}. For $u_1,\ldots,u_{k}$ denoting the children of an internal vertex $u$ in a stage of colour $c$, we define 
\begin{equation}
    I_{c}= \langle p_{[u_i]}p_{[v]}-p_{[u]}p_{[v_i]} \mid \text{ for all vertices } u, v \text{ in stage } c,\ i=1,\ldots, k \rangle.
\end{equation}
The ideal of model invariants is the ideal $I_\T=\sum_c I_c$. This ideal is contained in $\kerf$ and is easier to study than the staged tree's prime ideal because it has a concrete generating set. It is related to the model we are interested in by the following result.
\begin{thm}\cite[Theorem~3, Corollary~8]{DG} \label{lm:modelinvariantprime}
The staged tree model $\mathcal{M}_\T$ is equal to the variety of $I_{\T}$ intersected with the probability simplex $\Delta^{\circ}_{n-1}$. Moreover,  $\kerf$ is a minimal prime for $I_{\cT}$. \end{thm}
In order to discover whether a staged tree model has toric structure, we now search for binomial ideals between the ideal of model invariants $\It$ and $\kerf$.
The following is a key result in this endeavour, searching for linear transformations of variables that provide binomial generators. 

\begin{thm} \label{thm:toricJ}
Let $J$ be an ideal in $\R[p]$ such that $I_\T \subseteq J\subseteq \ker\varphi_\T$.  Let $\ell_1,\dots, \ell_n$ be  linear forms in $\R[p]$ such that 
\begin{enumerate}[i.]
\item $\ell_1,\dots, \ell_n$ are linearly independent, 
    \item each $\varphi_\T(\ell_i)$ can be represented by a monomial in $\R[\Theta,z]/\langle\theta-z \rangle$, and
    \item the ideal $J$ is generated by binomials in the new variables $\ell_1,\dots, \ell_n$.
\end{enumerate}
Then $\ker\varphi_\T$ is a toric ideal in the new variables $\ell_1,\dots, \ell_n$.
\end{thm}

\begin{proof}
For each $i=1, \ldots, n$ let $m_i$ be a monomial in $\R[\Theta,z]$ such that $\varphi(\ell_i)$ can be represented by $m_i$ in $\R[\Theta,z]/\langle \theta -z \rangle$. We define the monomial map $\psi:\R[p] \to \R[\Theta,z]$ by $\psi(\ell_i)=m_i$. Then  clearly $\ker \psi \subseteq \ker \varphi$, and we shall prove that $J \subseteq \ker \psi$. 
Take the binomial $\ell_{i_1}\! \cdots \ell_{i_s}-a\ell_{j_1}\! \cdots \ell_{j_s}$ in $J$. As $J \subseteq \ker \varphi$ the binomial $f= m_{i_1} \cdots m_{i_s}-am_{j_1} \cdots m_{j_s} $ is in $ \langle \theta -z \rangle$.  We want to prove that $f$ is in fact the zero polynomial. Let $g =f/\gcd(m_{i_1} \cdots m_{i_s},m_{j_1} \cdots m_{j_s})$. Since $\langle \theta -z \rangle$ is generated by linear forms it is a prime ideal, and it does not contain any monomials. It follows that $g$ is in $\langle \theta -z \rangle$. It remains to prove that $g=0$. If $g$ is not the zero polynomial, then one of its terms is divisible by a factor, say $\theta_1$. We can find  a point in the variety of $\langle \theta -z \rangle$ where $\theta_1$ is $0$ and all other coordinates are non-zero. But $g$ evaluated in this point will not be zero, which contradicts $g$ being in $\langle \theta -z \rangle$. Hence the only option is that $g$ is the zero polynomial. We have now concluded that $J \subseteq \ker \psi \subseteq \ker \varphi$. Since $ \ker \varphi$ is a minimal prime of $J$ and $\ker \psi$ is a prime ideal, we must have $\ker \psi = \ker \varphi$. If one uses $\ell_1, \ldots, \ell_n$ as variables, $\ker \psi$ is a toric ideal.
\end{proof}

In what remains of this section we discuss the \emph{ideal of minors} for a staged tree model $\T$ as a candidate for $J$ in Theorem \ref{thm:toricJ}. Let $c$ be a stage and let $k$ be the number of children for a vertex of this stage. We define the \emph{stage matrix} $M_c$ to be the matrix with columns $[p_{[u_1]},\dots, p_{[u_{k}]}]^{\top}$ for each vertex $u$ in the stage $c$. Let $J_c$ be the determinantal ideal generated by all $2 \times 2$-minors  of $M_c$. The generators of $J_c$ has the form $p_{[u_i]}p_{[v_j]}-p_{[u_j]}p_{[v_i]}$. The ideal of minors $J_\T$ for the staged tree $\T$ is the sum of ideals $J_c$ for all stages $c$ of $\T$. 
Authors in \cite{DG} refer this ideal as the  \emph{ideal of paths}. They show in  \cite[Lemma~9]{DG} that $J_\T$ is included in $\ker\varphi_\T$ and contains the ideal of model invariants. These results combined with \Cref{lm:modelinvariantprime} give the following.
\begin{lem}
For a staged tree $\T$, we have the chain of inclusions $I_\T\subseteq J_\T\subseteq \ker\varphi_\T$, and the prime ideal  $\kerf$ is a minimal prime for the ideal of minors $J_\T$.
\end{lem}

The fact that $J_\T$ is a determinantal ideal allows us to choose from many different generating sets via elementary row and column operations. 

\begin{lem}
\label{lm:elementaryoperations}
Elementary row and column operations on a matrix do not change its associated determinantal ideals. 
\end{lem}

This lemma is a well known property of determinantal ideals. For completeness we include a proof for the case of $2 \times 2$-minors. The proof for $t\times t$ minors with $t>2$ is analogous. 

\begin{proof}
Denote by $m_{ij}$ the $(i,j)$-entry of a matrix $M$ with entries from some polynomial ring $R$, and by $M_{\{i,j\}\times\{k,l\}}$ the $2\times2$ minor of $M$ that uses rows $i<j$, and columns $k<l$. 
It is enough to prove that the determinantal ideal does not change when one substitutes row $m_1$ by a linear combination of all the rows $a_1m_1+\dots +a_nm_n$ with $a_1\neq 0$. Call the second matrix $M'$. All the $2\times 2$-minors not involving the first row in $M'$ are the same as the respective ones in $M$, so $M'_{\{i,j\}\times\{k,l\}}=M_{\{i,j\}\times\{k,l\}}$ for $i\neq 1$. The minors that involve the first row are linear combination of minors in $M$: 
\begin{equation}
M'_{\{1,j\}\times\{k,l\}}=\det \left(\begin{bmatrix}a_1m_{1k}+\cdots + a_nm_{nk} & a_1m_{1l}+\cdots +a_nm_{nl} \\ m_{jk} & m_{jl}\end{bmatrix}\right) =\sum\limits_{i=1}^na_i\det \left( \begin{bmatrix} m_{ik} & m_{il}\\ m_{jk} & m_{jl}
\end{bmatrix}\right)=0.
\end{equation}
Since all but the first minor in the sum are already minors in $M'$ up to possibly a change of sign, minors $M'_{\{1,j\}\times\{k,l\}}$ can substitute $M_{\{1,j\}\times\{k,l\}}$ as generators in the determinantal ideal of $M$.
\end{proof}

The main idea is to apply appropriate row and column operations on each stage matrix $M_c$ that produce a total of exactly $n$ distinct entries in the final matrices. These entries will be the new variables $\ell_1,\ldots,\ell_n$ in \cref{thm:toricJ}. 

\begin{ex}\label{ex:minors}
Consider the tree in \cref{fig:minors}. This tree is not balanced. The ideal of minors $J_\T$ is generated by the $2\times 2$ minors of stage matrices 
\begin{equation}
M_{\color{Green}\bullet}=\begin{bmatrix}
p_1 & p_{10}\\
p_2 & p_{11}
\end{bmatrix}
\quad\text{and}\quad
M_{\color{magenta}\bullet}=\begin{bmatrix}
p_3 & p_6& p_1+p_2& p_1+\cdots+ p_8\\
p_4 & p_7 & p_3+p_4+p_5 & p_9\\
p_5 &p_8 &p_6+p_7+p_8 & p_{10}+p_{11}
\end{bmatrix}.
\end{equation}
Substitute the second row of $M_{\color{Green}\bullet}$ by the sum of the two rows. In $M_{\color{magenta}\bullet}$ we first substitute the second row by the sum of its three rows, and then we substitute its new first column by the sum of its negative, the negative of the second column, and the third column. This results in the two matrices
\begin{equation}
M'_{\color{Green}\bullet}=\begin{bmatrix}
p_1 & p_{10}\\
p_1+p_2 &p_{10}+p_{11}
\end{bmatrix}
\ \text{and} \
M'_{\color{magenta}\bullet}=\begin{bmatrix}
p_1+p_2-p_3-p_6 & p_6& p_1+p_2& p_1+\cdots+ p_8\\
p_1+p_2 & p_6+p_7+p_8 & p_1+\cdots+ p_8 & p_1+\cdots+ p_{11}\\
-p_5+p_6+p_7 &p_8 &p_6+p_7+p_8 & p_{10}+p_{11}
\end{bmatrix}.
\end{equation}
Take $\ell_1=p_1$, $\ell_2=p_1+p_2$, $\ell_3=p_1+p_2-p_3-p_6$, $\ell_4=-p_5+p_6+p_7$, $\ell_6=p_6+p_7+p_8$, $\ell_7=p_8$, $\ell_8=p_1+\cdots+p_8$, $\ell_9=p_{10}$, $\ell_{10}=p_{10}+p_{11}$ and $\ell_{11}=p_{1}+\cdots+p_{11}$. One can check that $\ell_1,\dots,\ell_{11}$ are linearly independent. The ideal of minors is a binomial ideal in the new variables. Images of $\varphi_\T(\ell_i)$ are represented by monomials in $\R[\Theta,z]\slash \langle \theta-z\rangle$. For instance $\varphi_\T(\ell_1)=\theta_1^2\tau_1$, $\varphi_\T(\ell_2)=\theta_1^2z$, $\varphi_\T(\ell_3)=\theta_1^3$, and $\varphi_\T(\ell_{11})=z^3$. By \Cref{thm:toricJ}, the prime ideal of the staged tree $\T$ is toric in the new variables $\ell_1,\dots,\ell_{11}$.
\end{ex}


%

In the following section we apply \cref{thm:toricJ} to a concrete class of staged tree models with a nice combinatorial description.

\begin{figure}
\centering
\begin{minipage}[b]{0.45\linewidth}
\centering
\includegraphics{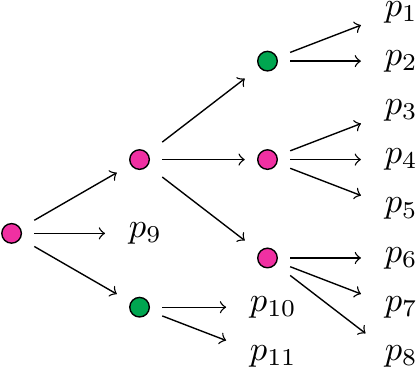}
\caption{A non-balanced staged tree which has a toric structure by Theorem \ref{thm:toricJ}.}\label{fig:minors}
\end{minipage}
\qquad
\begin{minipage}[b]{0.45\linewidth}
\centering
\includegraphics{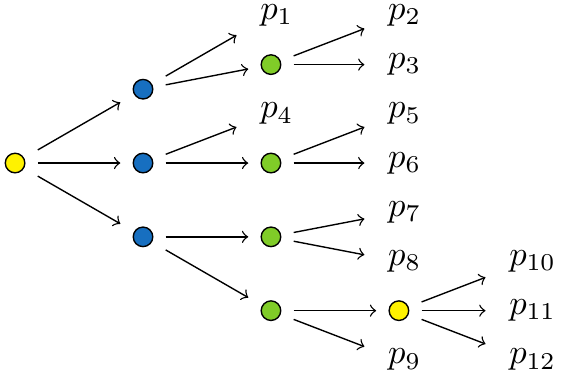}
\caption{An \sip-tree.\label{fig:siptree}}
\end{minipage}%
\end{figure}

\section{Extending the class of toric staged tree models}\label{sec:extend}


In this section we extend the class of staged tree models for which the ideal $\ker \varphi_{\T}$ is toric by allowing a linear change of variables. 

The trees we will study here are trees where (a subset of) their stages have the following property concerning induced subtrees. We say that a stage $c$ has the \emph{subtree-inclusion property}, for short \sip, if there is an index $i_c$ such that for every vertex $v$ of the stage $c$ each induced subtree $\T(v_i)$ contains a subtree identical to $\T(v_{i_c})$ with root $v_i$. A staged tree is henceforth called a \sip-tree if all of its stages satisfy the subtree-inclusion property.

An example of an \sip-tree is a staged tree where there is an $i$ such that the $i$-th child of each vertex is a leaf. In particular the tree in \cref{fig:biasedcoin} is \sip. The tree in \cref{fig:siptree} is SIP with $i_{\color{Yellow}\bullet}=1$ or $2$, $i_{\color{RoyalBlue}\bullet}=1$ and $i_{\color{LimeGreen}\bullet}=2$. 
If we have a \sip-tree, we can always redraw it so that $i_c=1$ for each stage $c$. Throughout this section we will assume all \sip-trees to be presented in this way.

Consider an \sip-tree $\T$. Our goal is to prove that the ideal $\ker \varphi_{\T}$ is toric after a linear change of variables. In order to do this we will introduce a stratified version of $\T$ which we denote $\T_{\str}$. The staged tree $\T_{\str}$ is identical to $\T$ as a tree, not considering edge-labels and colours of vertices. We let two vertices in $\T_{\str}$ be in the same stage if they lie on the same distance from the root, and are in the same stage in $\T$. The edges of $\T_{\str}$ are labelled accordingly. The stratified version $\T_{\str}$ of a SIP-tree is also SIP as the subtrees that are required to be identical in $\T$ always sit at the same distance from the root. The benefit of going via $\T_{\str}$ to study $\T$ is that there are no linear relations in the prime ideal of $\T_{\str}$. 
%
If $\Theta_{\str}$ denotes the set of labels in the tree $\T_{\str}$, and $\langle \theta-z \rangle_{\str} \subset \R[\Theta_{\str},z]$ the ideal of homogenised sum-to-one conditions, we have a canonical map $F: \R[\Theta_{\str},z]/\langle \theta-z \rangle_{\str} \to \R[\Theta,z]/\langle \theta-z \rangle$. In other words, we define $F$ so that the composition $\varphi_\T = F \circ \varphi_{\T_{\str}}$ holds. 

The key to proving that $\ker \varphi_{\T}$ is toric is to apply \cref{thm:toricJ} with $J$ being the ideal of minors $J_\T$. To find a choice of variables for which $J_\T$ is binomial we apply row and column operations to the matrices $M_c$, following Lemma \ref{lm:elementaryoperations}. The first step is to translate the SIP into a statement on the entries of the stage matrices $M_c$. 

\begin{lem}\label{lem:SIP_sub}
Let $\T$ be an SIP-tree, and let $\theta_1, \ldots, \theta_k$ be labels corresponding to a stage in the stratified tree $\T_{\str}$. Let $v\in \T_{\str}$ be a vertex such that the root-to-$v$ path goes through an edge labelled $\theta_1$. Then for each $i=1, \ldots, k$ there is a vertex $u$ such that $\frac{\theta_i}{\theta_1}\varphi_{\T_{\str}}(p_{[v]})=\varphi_{\T_{\str}}(p_{[u]})$. Moreover, if $v$ is not a leaf then $u$ and $v$ are in the same stage in $\T_{\str}$.   
\end{lem}
\begin{proof}
Let $w$ be the vertex where the path to $v$ takes the $\theta_1$-edge. Then $v$ belongs to the subtree $\T_{\str}(w_1)$. Since $\T$ is \sip, each induced subtree $\T_{\str}(w_i)$ contains a subtree identical to $\T_{\str}(w_1)$ with root $w_i$. So we can find a vertex $u$ in $\T_{\str}(w_i)$ where the path from $w_i$ to $u$ copies the steps of the path from $w_1$ to $v$. Then $\frac{\theta_i}{\theta_1}\varphi_{\T_{\str}}(p_{[v]})=\varphi_{\T_{\str}}(p_{[u]})$, and if $v$ is not a leaf then $u$ and $v$ are in the same stage. 
\end{proof}

The first outgoing edges of vertices in our SIP-tree $\T$ will play an important role in the proof of \cref{thm:SIP_toric}. For short we will refer to such edges as \emph{first-edges}. The next step is to prove that the linear forms that we will use as a new basis for $\R[p]$ are indeed $n$ linearly independent forms.  

\begin{lem}\label{lem:SIP_lin_indep}
Let $L$ be the set of monomials in $\R[\Theta_{\str},z]/\langle \theta-z \rangle_{\str}$ given by the root-to-leaf paths in $\T_{\str}$ and substituting every label of a first-edge by the variable $z$. Then $L$ lies in the image of the map $\varphi_{\T_{\str}}$, and the preimage of $L$ is a set of $n$ linearly independent linear forms.   
\end{lem}

\begin{proof}
Let $A$ be the set obtained from the monomials in $\R[\Theta_{\str},z]/\langle \theta-z \rangle_{\str}$ given by the root-to-leaf paths and substituting every label of a first-edge by the sum of the labels corresponding to that stage. It follows by repeated use of \cref{lem:SIP_sub} that each term in the expansion of such expression lies in the image of $\varphi_{\T_{\str}}$. On the other hand $A=L$, since the sum of the labels corresponding to a stage equals $z$. This proves that $L$ is in the image of $\varphi_{\T_{\str}}$. Moreover, the fact that the first parameter of each stage does not occur in the monomials in $\mathcal M$ makes the distinct monomials in $A$ linearly independent. Now, let's consider $\R[p]$ and $\R[\Theta_{\str},z]/\langle \theta-z \rangle_{\str}$ as $\R$-spaces. More precisely, we consider the subspaces $\R[p]_1=\spn(p_1, \ldots, p_n)$ and $\spn L$. Recall that there are no linear relations in $\ker \varphi_{\T_{\str}}$ as $\T_{\str}$ is squarefree. This is the same as saying that the restricted map $ \varphi_{\T_{\str}}:\R[p]_1 \to \spn L$ is bijective. But then the preimages of the elements of $A$ is just another choice of basis for $\R[p]_1$. 
\end{proof}

Now we are ready to prove that the all \sip-trees have toric structure. 

\begin{thm}\label{thm:SIP_toric}
If $\T$ is an \sip-tree then the associated ideal $\ker \varphi_\T$ is toric after a change of variables. The parametrisation revealing the toric structure is given by substituting the first parameter of each stage by the homogenising variable $z$ in the monomials 
\[
z^{d-d(\lambda_i)} \prod_{e \in E(\lambda_i)}\theta(e) \ \ \ \ \ \text{for }i=1, \ldots, n
\]
from \cref{eq:homphi}. Moreover, $p_1 + \dots +p_n$ is one of the new variables.
\end{thm}

\begin{proof}
For each stage $c$ of vertices in $\T$ we associate the matrix $M_c$ from which we get the ideal of minors $J_c$. Let $\varphi_{\T_{\str}}(M_c)$ denote the matrices obtained by applying $\varphi_{\T_{\str}}$ to the entries of $M_c$. The entries of $\varphi_{\T_{\str}}(M_c)$ are represented by monomials. The idea is to do row and column operations to $M_c$ resulting in a matrix $M_c'$ such that the entries of $\varphi_{\T_{\str}}(M_c')$ are monomials from the set $\mathcal{M}$ defined as in Lemma \ref{lem:SIP_lin_indep}. We point out that the matrices $M_c$ considered here do not necessarily define the ideals of minors of $\T_{\str}$. The matrix $\varphi_{\T_{\str}}(M_c)$ only serves as an intermediate step in understanding the effect of row and column operations on~$M_c$.

Recall that the columns in the matrix $M_c$ are indexed by the vertices in the stage $c$ in $\T$. The entries in the column of $\varphi_{\T_{\str}}(M_c)$ indexed by a vertex $v$ are 
\[\varphi_{\T_{\str}}(p_{[v_1]})=\theta_1\varphi_{\T_{\str}}(p_{[v]}), \ \varphi_{\T_{\str}}(p_{[v_2]})=\theta_2\varphi_{\T_{\str}}(p_{[v]}), \  \ldots , \  \varphi_{\T_{\str}}(p_{[v_k]})\theta_k\varphi_{\T_{\str}}(p_{[v]})\]
where $\theta_1, \ldots, \theta_k$ are the labels of the stage of $v$ in $\T_{\str}$. Suppose the root-to-$v$ path goes through exactly $m>0$ first-edges, and let
\[ \{ \tau_{11}, \tau_{12}, \ldots, \tau_{1k_1} \}, \{ \tau_{21}, \tau_{22}, \ldots, \tau_{2k_2} \}, \ldots, \{ \tau_{m1}, \tau_{m2}, \ldots, \tau_{mk_m} \} \]
be the sets of label of the corresponding $m$ stages in $\T_{\str}$, so that $\tau_{11} \tau_{21} \cdots \tau_{m1}$ divides $\varphi_{\T_{\str}}(p_{[v]})$. 
By repeated use of Lemma \ref{lem:SIP_sub}, we see that for each $m$-tuple $(j_1, \ldots, j_m)$ of integers $1 \le j_i \le k_i$ there is a vertex $u$ in the same stage as $v$ such that 
\begin{equation}
\frac{\tau_{1j_1}\tau_{2j_s} \cdots \tau_{mj_m}}{\tau_{11}\tau_{21}\cdots \tau_{m1}} \varphi_{\T_{\str}}(p_{[v]}) = \varphi_{\T_{\str}}(p_{[u]}).
\end{equation}
Since $u$ is also in the same stage as $v$ in $\T$ there is a column for $u$ in $M_c$. Consider the column operation that replaces the $v$-column by the sum of the $k_1 k_2 \cdots k_m$ columns of those $u$'s. Then the entry $p_{[v_i]}$ in row $i$ is replaced by the linear form $\sum_u p_{[u_i]}$ such that
\begin{align}\label{eq:prf_SIP}
\varphi_{\T_{\str}}\Big(\sum_u p_{[u_i]}\Big) &= \theta_i \sum_{u} \varphi_{\T_{\str}}(p_{[u]}) \\
&= \theta_i \frac{     (\tau_{11}+ \dots + \tau_{1k_1})( \tau_{21}+ \dots + \tau_{2k_2})(\tau_{m1}+ \dots + \tau_{mk_m})    }{\tau_{11}\tau_{21}\cdots \tau_{m1}} \varphi_{\T_{\str}}(p_{[v]})\\
& = \theta_i \frac{z^m}{\tau_{11}\tau_{21}\cdots \tau_{m1}}  \varphi_{\T_{\str}}(p_{[v]}).
\end{align}
In other words, all labels of first-edges in $\varphi_{\T_{\str}}(p_{[v]})$ were substituted by $z$. The idea is to do several such column operations, but we must be careful in which order to do them. Notice that the root-to-$u$ path, for each of the $u$'s we consider above, goes through at most $m$ first-edges. Moreover it goes through exactly $m$ first-edges only if $u=v$. So we can order the columns we want to operate on in weakly decreasing order according to their associated number $m$ of first-edges. In this way the columns that are used in each step are columns from the original matrix $M_c$. 

Let $M'_c$ be the resulting matrix when we have done the above operation to all columns where it can be applied. The entries of $\varphi_{\T_{\str}}( M'_c)$ are all of the form (\ref{eq:prf_SIP}). What remains is to eliminate the $\theta_1$'s in the first row. 
This is accomplished by the row operation replacing the first row by the sum of all rows. Let $M''_c$ denote the matrix obtained after this row operation. The entries $\varphi_{\T_{\str}}(M''_c)$ are monomials obtained from the entries of $\varphi_{\T_{\str}}(M_c)$ and substituting the labels of first-edges by $z$. Now remember that the entries of $M_c$ are all of the form $p_{[v]}$ for some vertex $v$. Consider the root-to-leaf path going through $v$ and after $v$ taking the first-edge in each step. Then take the product of the labels along this path, and substitute each label of a first-edge by $z$. This gives the corresponding entry in the matrix $\varphi_{\T_{\str}}(M''_c)$, and in this way we see that all entries of $\varphi_{\T_{\str}}(M''_c)$ belong to the set $\mathcal{M}$. By Lemma \ref{lem:SIP_lin_indep} the distinct entries of $M''_c$ are linearly independent linear forms $\ell_1, \ldots, \ell_n$. The 2-minors of $M''_c$ is then a set of binomials in $\ell_1, \ldots, \ell_n$ which is a generating set of the ideal $J_c$. It follows that $J_{\T}$ is generated by binomials in $\ell_1, \ldots, \ell_n$. By this the requirements of Theorem \ref{thm:toricJ} are fulfilled, and we have proved that $\ker \varphi_{\T}$ is toric after a linear change of variables.  Last, $p_1 + \dots + p_n$ is among the new variables as it is the preimage of $z^d$. 
\end{proof}

We have now seen two types of staged trees giving rise to toric ideals: the balanced trees and the SIP-trees. Next we shall see that we can combine these two classes to construct even more toric staged trees.

\begin{thm}\label{thm:extend_toric}
Let $\T$ be a staged tree, and let $\cS$ be a subtree with the same root as $\T$. Let $v_1, \ldots, v_m$ denote the vertices of $\T$ which are leaves of the subtree $\cS$. Suppose
\begin{enumerate}
\item the tree $\cS$ is balanced,
\item no internal vertex of $\cS$ is in the same stage as a vertex in some $\T(v_i)$,
\item the stages of vertices in $\T(v_1), \ldots, \T(v_m)$ satisfy the subtree-inclusion property. 
\end{enumerate}
Then $\T$ has toric structure. 
\end{thm} 
\begin{proof}
Let $\varphi_{\cS}: \R[q_1, \ldots, q_m] \to \R[\Theta,z]$ and $\varphi_{\T}:\R[p] \to \R[\Theta,z]$ be defined in the usual way. This means that $\varphi_{\cS}$ maps $q_i$ to the product of the edge labels along the path from the root to $v_i$, or in other words $\varphi_{\cS}(q_i)=\varphi_{\T}(p_{[v_i]})$. Let $\overline{\ker \varphi_{\cS}}$ denote the ideal of $\R[p]$ obtained from $\ker \varphi_{\cS}$ by substituting each $q_i$ by $p_{[v_i]}$. As $\cS$ is a balanced staged tree the ideal $\ker \varphi_{\cS}$ is binomial. Then $\overline{\ker \varphi_{\cS}}$ is generated by binomials in $p_{[v_1]}, \ldots, p_{[v_m]}$. We construct the ideal $\overline{I_{\cS}} \subset \R[p]$ from $I_{\cS} \subset \R[q_1, \ldots, q_m]$ in the same way. The ideal $\overline{I_{\cS}}$ is exactly the same as taking the sum of the ideals $I_c$ over all stages $c$ of $\T$ that occurs in the subtree $\cS$. 

Let $C$ be the set of stages that occur in the induced subtrees $\T(v_1), \ldots, \T(v_m)$. For a stage $c \in C$ the matrix $M_c$ is a $1 \times m$ block matrix where the $i$-th block is the matrix of the stage $c$ in $\T(v_i)$. We can do the column operations from the proof of Theorem \ref{thm:SIP_toric} to each of the blocks. We can also do the row operation of replacing the first row by the sum of all rows simultaneously on all blocks. It then follows in the same way as in the proof of Theorem \ref{thm:SIP_toric} that $\sum_{c \in C} J_c$ is a binomial ideal in linearly independent linear forms $\ell_1, \ldots, \ell_n$. Among $\ell_1, \ldots, \ell_n$ we find $p_{[v_1]}, \ldots, p_{[v_m]}$. 

The idea is now to apply Theorem \ref{thm:toricJ} with the ideal $J=\overline{\ker \varphi_{\cS}}+\sum_{c\in C}J_c$. We have seen that $J$ is generated by binomials in $\ell_1, \ldots, \ell_n$, and it is clear that $J \subseteq \ker \varphi_{\T}$. We also have the inclusion $I_{\T} \subseteq J$ as
\[
I_{\T} = \overline{I_{\cS}} + \sum_{c \in C} I_c \subseteq \overline{\ker \varphi_{\cS}}+\sum_{c\in C}J_c=J
\]
which completes the proof. 
\end{proof}

\begin{rmk}
The balanced trees are covered by Theorem \ref{thm:extend_toric}, as we can take $\cS$ to be the whole tree. Theorem \ref{thm:extend_toric} also covers \sip-trees as we can take $\cS$ as just the root. 
\end{rmk}

\begin{rmk}
Even though the Theorems~\ref{thm:SIP_toric} and~\ref{thm:extend_toric} do not give explicit expressions for the linear change of variables revealing the toric structure of $\kerf$, the proofs are constructive. Given any staged tree $\T$ covered by \cref{thm:extend_toric}, the linear change of variables can be obtained following the steps in the two proofs. 
\end{rmk}

\begin{ex}\label{ex:continued}
We now return to the tree $\T$ in \cref{fig:nonCM} discussed in Example \ref{ex:nonCM}. Let $v_1, \ldots, v_8$ be the vertices on distance 3 from the root, and let $\cS$ be the subtree with $v_1, \ldots, v_8$ as leaves. Then $\cS$ is balanced, and the light green and dark green stages in $\T(v_1), \ldots, \T(v_8)$ satisfy the subtree-inclusion property as all children of the green vertices are leaves. So we can apply Theorem \ref{thm:extend_toric}, and it follows that $\ker \varphi_\T$ is toric. The ideal $\overline{\ker \varphi_{\cS}}$ defined as in the proof of Theorem \ref{thm:extend_toric} is generated by binomials in the linear forms 
\begin{equation}
\{ p_{[v_1]}, \ldots, p_{[v_8]}\} = \{  
p_1+p_2,\ p_3,\ p_4+p_5,\ p_6,\ p_7+p_8,\ p_9, \ p_{10}+p_{11},\ p_{12} \}.
\end{equation}
The ideals $J_{\color{LimeGreen}\bullet}$ and $J_{\color{OliveGreen}\bullet}$ are given by the determinants
\begin{equation}
\begin{vmatrix}
p_1 & p_{10} \\
p_2 & p_{11}
\end{vmatrix}=\begin{vmatrix}
p_1+p_2 & p_{10}+p_{11} \\
p_2 & p_{11}
\end{vmatrix} \ 
\text{ and } \ 
\begin{vmatrix}
p_4 & p_{7} \\
p_5 & p_{8}
\end{vmatrix}=\begin{vmatrix}
p_4+p_5 & p_{7}+p_{8} \\
p_5 & p_{8}
\end{vmatrix}.
\end{equation}
Altogether, $\ker \varphi_\T$ is a binomial ideal if we choose $p_1+p_2,\ p_2, \ p_3,\ p_4+p_5,\ p_5,\ p_6,\ p_7+p_8,\ p_8,\ p_9, p_{10}+p_{11},p_{11}, p_{12}$ as new variables. 
\end{ex}

\section{One-stage tree models}
\label{sec:one_stage}

This section explores the algebraic structure of staged trees with all inner vertices in the same stage, called \emph{one-stage trees}.  We  connect staged tree models represented by binary one-stage trees to Veronese embeddings---well known maps in algebraic geometry and commutative algebra \cite{cox2011toric,bocci2019introduction,Harris}. We further show that all binary one-stage tree models have toric structure and give evidence of this toric structure in other classes. 
\medskip

Take positive integers $k$ and $d$ and let $\R[\Theta]$ be the polynomial ring in variables $\theta_1,\dots,\theta_k$.  The Veronese embedding associated to  $k$ and $d$ is defined as the map
\begin{equation}
\label{eq:veroneseEmbedding}
  \nu_{k,d}\colon \R[q_{\boldsymbol{i}} \mid \boldsymbol{i}\in \Z_{\geq 0}^{k}, i_1+\cdots +i_k=d]\rightarrow \R[\Theta],\quad \nu_{k,d}(q_{\boldsymbol{i}})=\theta^{\boldsymbol{i}}.
\end{equation}
The  kernel of $ \nu_{k,d}$ is the toric ideal of the  Veronese embedding.
 It is generated by quadratics  $q_{\boldsymbol{i}}q_{\boldsymbol{j}}=q_{\boldsymbol{s}}q_{\boldsymbol{t}}$ such that $\boldsymbol{i}+\boldsymbol{j}=\boldsymbol{s}+\boldsymbol{t}$. 
The algebra $\imv$ in $\R[\Theta]$ over  $\R$ is the \emph{Veronese algebra} for positive integers $k$ and $d$.  We will use this notion extensively in this section. The set with the $\binom{k+d-1}{d}$ monomials $\Theta^d={\{\theta^{\boldsymbol{i}}\mid i\in \Z_{\geq 0}^{k}, i_1+\cdots +i_k=d\}}$ is the \emph{canonical basis} for the Veronese algebra $\imv$.

Denote henceforth by $\mathfrak{T}_{k,d}$ the set of  all one-stage trees of depth $d$ with  parameters $\theta_1,\dots,\theta_k$.  The staged trees with the lowest number of edges in $\mathfrak{T}_{k,d}$ are the ones that have the form of a \emph{caterpillar}, that is they have exactly one internal node at each level.  The symbol $\cC$ is now used to indicate such a tree.  The trees in $\mathfrak{T}_{k,d}$  with the highest number of edges are the ones where all root-to-leaf paths have depth $d$. These are called \emph{maximal} trees.
Even though one-stage tree models have this strong restriction of having exactly one stage, they still provide a variety of models not covered by previous results. For instance, only one-stage maximal trees where all leaves are at the same distance from the root are balanced.  And the staged tree in \cref{fig:monomial_alg_ex} as well as trees $\T_8, \T_{10},\T_{11}, \T_{14},\T_{18}$ in \cref{fig:tableT33} do not satisfy the subtree-inclusion property of \cref{thm:SIP_toric}.

\begin{lem}
Maximal trees are the only balanced one-stage tree models in $\mathfrak{T}_{k,d}$. 
\end{lem}

\begin{proof}
Let $u_1,\dots, u_k$ be children of an internal vertex $u$ of the one-stage tree $\cT$. If the tree is maximal then $t(u_i)=t(u_j)$ for all children $u_i$ and $u_j$ of $u$. Hence,  the equality $t(u_i)t(v_j)=t(u_j)t(v_i)$ is true for any two internal vertices in the maximal one-stage tree. 
If the tree is not maximal, there must be a vertex $u$ in $\cT$ that has both a leaf  $u_i$ and another internal vertex $u_j$ as its children. Take  any vertex $v$ in $\T$ with all leaves as its children. The sides of the equation  $t(u_i)t(v_j)= t(v_i)t(u_j)$ are then of different degrees, making equality impossible.
\end{proof}

Given a one-stage tree $\T$ and one of its root-to-leaf paths~$\lambda$,  let $\boldsymbol{i}_{\lambda}$  be the vector in $\Z_{\geq 0}^k$ recording the number of times $\theta_1,\dots,\theta_k$ appear in  $\lambda$. 
The algebra associated to  $\cT$ has the form 
\begin{equation}
\imf= \R[\theta^{\boldsymbol{i}_{\lambda}} \mid \lambda \text{ root-to-leaf path in } \T]\slash \langle\theta-z\rangle.
\end{equation}
\begin{lem}\label{lm:inclusiontoVeronese}
The algebra of any one-stage tree model in  $\mathfrak{T}_{k,d}$ is a subalgebra of the Veronese algebra $\imv$.
\end{lem}

\begin{proof}
Take  $\boldsymbol{i}_{\lambda}=(i_1,\dots, i_k)$ as above.  
By the multinomial theorem, each generator $\theta^{\boldsymbol{i}_{\lambda}}z^{d-i_1-\cdots -i_k}=\theta^{\boldsymbol{i}_{\lambda}}(\theta_1+\cdots +\theta_k)^{d-i_1-\cdots -i_k}$ in the image of $\varphi_{\T}$ is a linear combination of the generators of $\imv$.
\end{proof}

\begin{rmk}
The involvement of Veronese algebras in both the definition of a Multinomial distribution and in the one-stage tree models allows us to draw a natural connection between these two concepts. This extends the well-known connection of squarefree staged trees to these distributions~\cite{Goergen.etal.2020}. Recall that a Multinomial distribution models the outcome of $d$ repeated experiments where the outcome of each trial has a categorical distribution with state space $\Omega$ of size $k$ and associated probabilities $\theta_1,\dots, \theta_k$.  The probability that in $d$ trials, the $i^\text{th}$ output in $\Omega$ appears exactly  $t_i$  times, for  $i=1,\dots, k$, is $\binom{d}{t_1,\dots,t_k}\theta_1^{t_1}\dots \theta_k^{t_k}$ when $t_1+\cdots +t_k=d$ and $0$ otherwise. These are all terms in the sum of  root-to-leaf parametrisations in the maximal tree of $\mathfrak{T}_{k,d}$.  Non-maximal one-stage trees can be interpreted as realisations of a generalisation of a Multinomial distribution which involves conditions. The terms in the sum of all root-to-leaf parametrisations, again, model the probabilities of the possible outcomes. Compare for instance the example of flipping a biased coin from the introduction of this present paper to subsection on Bernoulli binary models in \cite{bocci2019introduction}. 
\end{rmk}
Returning to one-stage trees, the following is an immediate consequence of \cref{lm:inclusion} and~\ref{lm:inclusiontoVeronese} and \cref{prop:monomial_algebra}.
\begin{cor}\label{cor:fullVeronese}
 Let $\T\subset \T'$  be two trees in $\mathfrak{T}_{k,d}$ that share the same root. If   $\imf=\imv$ then  $\im \: \varphi_{\T'}$ is equal to $\imv$, too. In particular, all such trees have toric structure. 
\end{cor}

In what follows we show that all one-stage trees in $\mathfrak{T}_{2,d}$  have toric structure.  We do this by first proving that any caterpillar tree in  $\mathfrak{T}_{2,d}$  has the Veronese algebra $\im\: \nu_{2,d}$ as its algebra. \cref{cor:fullVeronese}  then completes the proof. 

\begin{thm}
\label{thm:binary}
All binary one-stage trees have toric structure and are  parametrised  by  $\im \, \nu_{2,d}$.
\end{thm}

\begin{proof}
Take a caterpillar tree $\cC$ of depth $d$  formed by cutting off branches of a staged  tree $\cT$ in $\mathfrak{T}_{2,d}$.  The strategy is to prove inclusions \(\im \,\nu_{2,d}\subseteq\im \,\varphi_{\cC}\subseteq\imf \subseteq \im \,\nu_{2,d}.  \) We only need to prove the first inclusion as the other two hold by \cref{lm:inclusion} and~\ref{lm:inclusiontoVeronese}. 

Use $\theta(i)$ and $\bar{\theta}(i)$ in $\{\theta_1,\theta_2\}$ to denote the parameters in edges emanating from the $i$-th internal vertex of the caterpillar $\cC$, where $\bar{\theta}(t)$ is in the edge leading to a leaf and $i=1,\dots,d$; see \cref{fig:cat}.
%
\begin{figure}
    \centering
    \includegraphics[scale=1]{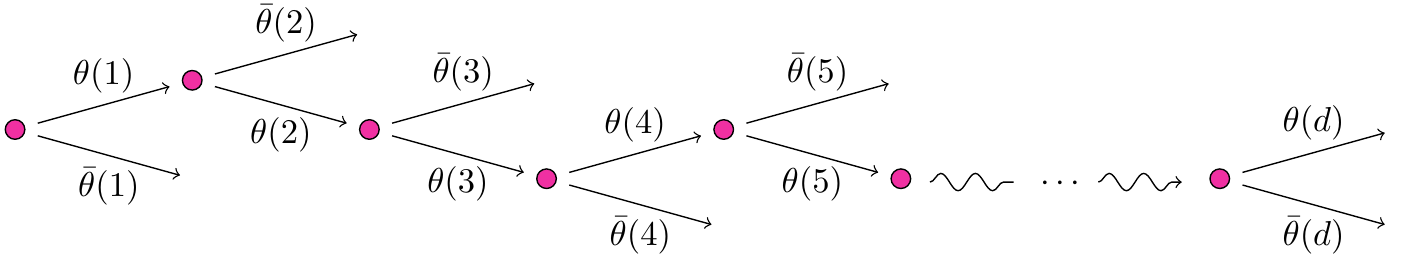}
    \caption{A caterpillar tree in $\mathfrak{T}_{2,d}$. Here we have $\theta(1)=\theta_1, \, \bar{\theta}(1)=\theta_2, \, \theta(2)=\theta_2, \, \bar{\theta}(2)=\theta_1$, and so on.   }
    \label{fig:cat}
\end{figure}
 For every $t$ between $0$ and $d$, let $f_t= \theta(1)\cdots \theta(d-t)z^{t}$ and $\bar{f}_t= \theta(1)\cdots \theta(d-t-1)\bar{\theta}(d-t)z^t$.  In particular, $f_d=\bar{f}_d=z^d$. The form $\bar{f}_t$ changes from $f_t$ only by the use of $\bar{\theta}(d-t)$ instead of $\theta(d-t)$.  The collection $\{f_t,\bar{f}_t \mid t=0,\cdots,d\}$  includes all root-to-leaf parametrisations in tree $\cC$, and hence serves as a generating set for the algebra $\im \: \varphi_\cC$.  Since $z$ is the sum $\theta_1+\cdots +\theta_k$, the terms above can be rewritten as sums
\begin{align}
f_{t}  &=\sum\limits_{i=0}^t \binom{t}{i}\theta(1)\cdots {\theta}(d-t) \cdot \theta_1^i\theta_2^{t-i}= \sum\limits_{i=0}^t \binom{t}{i}f_{t,i}\label{eq:fnonbars}
\intertext{and}
\bar{f}_{t}  &=\sum\limits_{i=0}^t \binom{t}{i}\theta(1)\cdots \theta(d-t-1)   \bar{\theta}(d-t) \cdot \theta_1^i\theta_2^{t-i}=\sum\limits_{i=0}^t \binom{t}{i}\bar{f}_{t,i}.\label{eq:fbars}
\end{align}
For each $t$,  take $F_t= \{f_{t,i}\mid i=0,\ldots, t\}$ to be the support of  $f_t$. Similarly, take the set $\bar{F}_t= \{\bar{f}_{t,i}\mid i=0,\dots, t\}$.  Notice that when $t$ is not zero, all terms of $f_t$ appear in either $f_{t-1}$ or $\bar{f}_{t-1}$. Therefore,  $F_t$ is the union of sets  $F_{t-1}$ and $\bar{F}_{t-1}$.
  
We will show that each  $F_t$ is included the image of  $\varphi_\cC$.  The set $F_0$ contains only one element, namely the term $f_0$ which is one of the generators of $\im \: \varphi_\cC$.   Assume this is true for the set $F_{t-1}$. Since $F_{t}=F_{t-1}\cup \bar{F}_t$, it is enough to show that $\bar{F}_t$ is in the image of $ \varphi_\cC$. 
The set $\bar{F}_0$ is in the image of $ \varphi_\cC$ since its only element is the generator $\bar{f}_{0}$.  Take now $\bar{F}_t$ for some $t$ greater than zero. Then all but one term of $\bar{F}_t$  belong to $F_t$. Indeed, if $\theta(d-t)$ is $\theta_1$ then $\bar{f}_{t,i}=f_{t,i-1}$, for all $i=1,\dots, t$, and if $\theta(d-t)=\theta_2$, then $\bar{f}_{t,i}=f_{t,i+1}$,  for all $i=0,\dots, t-1$. In both cases the remaining term is in the image of  $\varphi_\cC$ due to the sum in  \cref{eq:fbars} being in $\im \, \varphi_\cC$. 
In particular, the canonical generating set  $F_d$ of $\im \, \nu_{2,d}$, is in the image of $\varphi_\cC$. Since all three inclusions are achieved, \cref{prop:monomial_algebra} concludes that all binary one-stage tree models have toric structure. 
\end{proof}
\begin{rmk}
Example  \ref{ex:algebra_binary} suggest how we can use the monomial algebra $\im\, \nu_{2,d}$ of a binary one-stage tree to obtain linear transformations of variables that make $\kerf$ toric: search for $n$  linearly independent linear forms $\ell_1,\dots,\ell_n$ , where $n$ is the number of leaves, that map to monomials $\Theta^d$ of $\R[\theta_1,\theta_2]$. The resulting ideal is generated by binomials of degree at most two.
\end{rmk}

Unfortunately, not all one-stage trees have Veronese algebras as their defining algebras. Take for instance any caterpillar tree in $\mathfrak{T}_{k,d}$ for $k$ and $d$ greater than two. For the algebra of a one-stage tree model to be all $\imv$, the tree  must contain at least  $\binom{k+d-1}{d}$ leaves. The $kd-d+1$ leaves of a caterpillar tree with $k,d>2$ do not suffice.  Even when a one-stage tree has enough leaves, some of them may contribute to linear relations, and the algebra $\imf$ can still be a proper subalgebra of $\imv$.  Such an example is any tree in  the final row of \cref{fig:tableT33}. However, their algebras still appear to  be  monomial algebras generated by subsets of the canonical basis of the Veronese algebra  $\imv$. 
\begin{thm}\label{thm:T33}
All three-parameter one-stage tree models of depth at most three have toric structure.
\end{thm}
\begin{proof}
The  only staged tree in $\mathfrak{T}_{3,1}$ has  $\kerf$ equal to the zero ideal.  Every algebra in $\mathfrak{T}_{3,2}$ is isomorphic to the algebra of one of the three staged trees in \cref{fig:32} up to a permutation of variables~$\theta_1,\theta_2$ and~$\theta_3$.  All of the three trees in \cref{fig:32} satisfy the subtree-inclusion property of \cref{thm:SIP_toric}.  In particular, the algebra of the second staged tree shown in \cref{fig:32} is all $\im \, \nu_{3,2}$. because $\theta_3^2=\theta_3z-\theta_1\theta_3-\theta_2\theta_3$ is in $\im \,\varphi_{T_3}$, providing further evidence of toric structure.

\begin{figure}
\centering
\begin{subfigure}[b]{0.3\textwidth}
\centering
\includegraphics{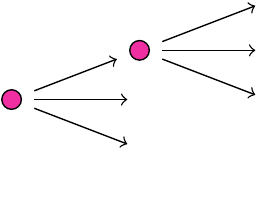}
\end{subfigure}
\quad
\begin{subfigure}[b]{0.3\textwidth}
\centering
\includegraphics{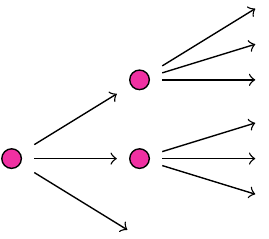}
\end{subfigure}
\quad
\begin{subfigure}[b]{0.3\textwidth}
\centering
\includegraphics{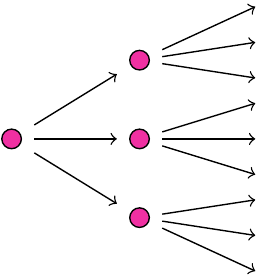}
\end{subfigure}
\caption{The three one-stage trees in $\mathfrak{T}_{3,2}$ up to a permutation of parameters. The edges have labels $\theta_1,\theta_2,\theta_3$, respectively, read from top to bottom. They are  all SIP trees.\label{fig:32}}
\end{figure} 

The collection $\mathfrak{T}_{3,3}$ is quite large. As the remainder of this proof will show, it will be enough to prove that the $20$ trees in \cref{fig:tableT33} have toric structure.  

\begin{figure}[t]
    \centering
    \includegraphics[width=\textwidth]{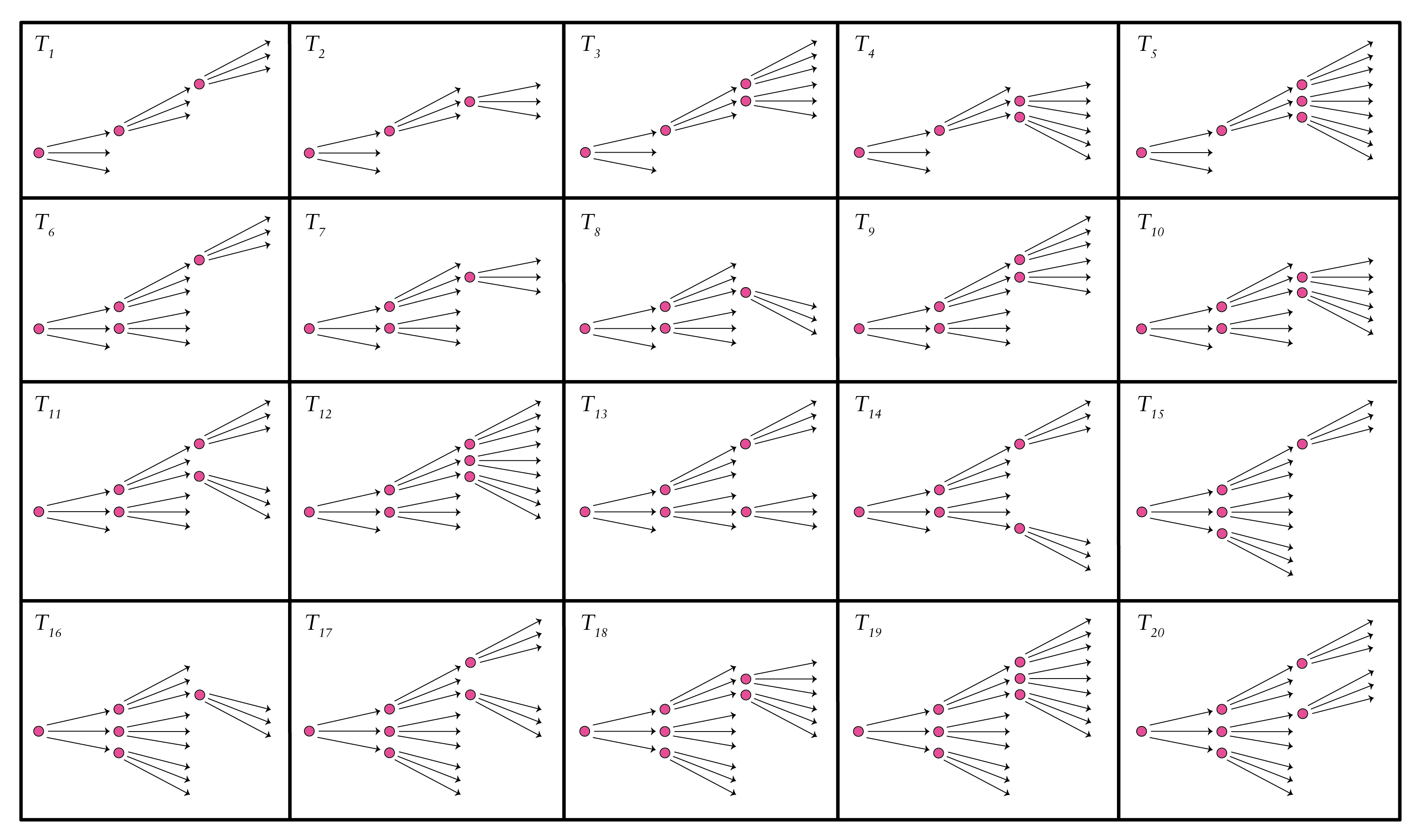}
    \caption{Staged trees in $\mathfrak{T}_{3,3}$ with labels $\theta_1,\theta_2,\theta_3$, respectively, read from top to bottom for each stage. The trees $\T_{13}$ and $\T_{14}$ are defined by the Veronese algebra $\im\,\nu_{3,3}$. The algebra of any other one-stage tree in $\mathfrak{T}_{3,3}$ are either obtained by permuting parameters of a tree in this table, or up to a permutation of parameters contain one of the trees $\T_{13}$ or  $\T_{14}$.}
    \label{fig:tableT33}
\end{figure}

All but the trees numbered $\T_4,\T_8,\T_{10},  \T_{11},\T_{14}$, and $\T_{18}$ in this table satisfy the subtree-inclusion property and hence are toric by \cref{thm:SIP_toric}.  Even more, the algebra of $\T_{13}$ is the full Veronese algebra $\im \, \nu_{3,3}$.   To realise this, compare the generating set of $\T_{13}$  obtained by all root-to-leaf paths with the canonical basis of $\im \, \nu_{3,3}$. The  monomials $\theta_1\theta_2\theta_3,  \theta_1\theta_3^2, \theta_2\theta_3^2$ and $\theta_3^3$  in this basis do not coincide with any root-to-leaf path parametrisation of $\T_{13}$. However, they all are linear combinations of the latter ones: $ \theta_1\theta_2\theta_3= \theta_1\theta_2z-\theta_1^2\theta_2-\theta_1\theta_2^2,    \quad \theta_1\theta_3^2= \theta_1\theta_3z-\theta_1^2\theta_3-\theta_1\theta_2\theta_3,  \quad\theta_2\theta_3^2= \theta_2\theta_3z-\theta_1\theta_2\theta_3-\theta_2^2\theta_3, $ and $  \theta_3^3= \theta_3z^2-\theta_1\theta_3z-\theta_2\theta_3z-\theta_1\theta_3^2-\theta_2\theta_3^2$. By \cref{lm:inclusiontoVeronese}, all one-stage trees in $\mathfrak{T}_{3,3}$ that contain $\T_{13}$ up to a permutation of parameters have toric structure. 

By \cref{lm:inclusion},  the algebras of staged trees $\T_4,\T_8,\T_{10}, \T_{11}$, and $\T_{18}$ are included in the monomial algebras of $\T_5,\T_{16}, \T_{18}, \T_{17}$,  and $\T_{19}$, respectively. We will show that the inclusions are equalities, by comparing their root-to-leaf parametrisations. For $\T_4$,  its root-to-leaf parametrisation changes from the one of $\T_5$  only by the absence of $\theta_1^3$. However, $\theta_1^3$ is equal to the linear combination  $\theta_1^2z-\theta_1^2\theta_2z-\theta_1^2\theta_3z$  of  elements in $\im \, \varphi_{\T_4}$, read by root-to-leaf paths.
Via a similar argument, $\theta_1^3$ is equal to $\theta_1^2z-\theta_1^2\theta_2z-\theta_1^2\theta_3z$  in  $\im \, \varphi_{\T_{18}}$, which makes the latter one equal to $\im \: \varphi_{\T_{19}}$. 
In $\T_8$, the only forms in $\im \, \varphi_{\T_{16}}$ not appearing in its root-to-leaf paths are  $\theta_1\theta_3z$ and $\theta_3^2z$. These can be expressed as  linear combinations of forms read from root-to-leaf paths  $\theta_1\theta_3z=-\theta_1^2\theta_3-\theta_1\theta_2\theta_3-\theta_1\theta_3^2$, and  $\theta_3^2z=\theta_{3}z^2-\theta_1\theta_3z-\theta_2\theta_3z$ in $\im \, \varphi_{\T_{8}}$.  The latter equality serves to prove that   $\im \,\varphi_{\T_{10}}=\im \,\varphi_{\T_{18}}$ and  $\im \,\varphi_{\T_{11}}=\im \,\varphi_{\T_{17}}$.  By \cref{lm:permuting}, the staged trees $\T_4,\T_8,\T_{10}, \T_{11}$, and  $\T_{18}$ have toric structures. 

Lastly, the  image of  $\T_{14}$ is all the Veronese algebra $\im \, \nu_{3,3}$. Similarly to $\T_{13}$,  the four monomials $\theta_1\theta_2^2,\theta_1\theta_3^2,\theta_2^3$, and $  \theta_3^3$ in the canonical basis  of $\im \, \nu_{3,3}$ that do not appear in the root-to-leaf parametrisation of $\T_{14}$ are linear combinations of  the latter ones: 
 $ \theta_1\theta_2^2= \theta_1\theta_2z-\theta_1^2\theta_2-\theta_1\theta_2\theta_3,
    \quad \theta_1\theta_3^2= \theta_1\theta_3z-\theta_1^2\theta_3-\theta_1\theta_2\theta_3,
   \quad\theta_2^3= \theta_2^2z-\theta_1\theta_2^2-\theta_2^2\theta_3, $ and $
   \theta_3^3= \theta_3z^2-\theta_1\theta_3z-\theta_1\theta_2\theta_3-\theta_2^2\theta_3-\theta_2\theta_3^2-\theta_1\theta_3^2-\theta_2\theta_3^2 $.  Due to \cref{lm:inclusiontoVeronese}, all one-stage trees in  $\mathfrak{T}_{3,3}$ containing $\T_{14}$ up to a permutation of parameters have toric structure.  
   
Any  tree in $\mathfrak{T}_{3,3}$  that does not appear in \cref{fig:tableT33} either contains $\T_{13}$ or $\T_{13}$  up to a permutation of parameters, or is equal to one of the trees in the table up to such a permutation.  \cref{lm:inclusiontoVeronese} and \cref{lm:permuting} thus conclude the proof.
\end{proof}

In the proof of \Cref{thm:T33} techniques developed throughout this paper come together. These techniques can prove the toric structure of many other one-stage tree models.  For instance, the first two trees in \cref{fig:34} have toric structure, the second one having $\im\,\nu_{3,4}$ as its defining algebra. There are however still one-stage trees that provide challenges. The smallest such tree is the caterpillar in \cref{fig:t34c}.

\begin{figure}
\centering
\begin{subfigure}[b]{0.3\textwidth}
\centering
\includegraphics{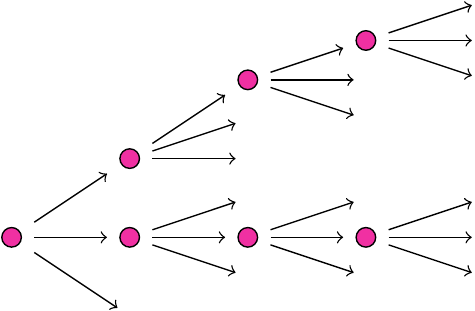}
\caption{This is an SIP tree. \label{fig:t34a}}
\end{subfigure}
\quad
\begin{subfigure}[b]{0.3\textwidth}
\centering
\includegraphics{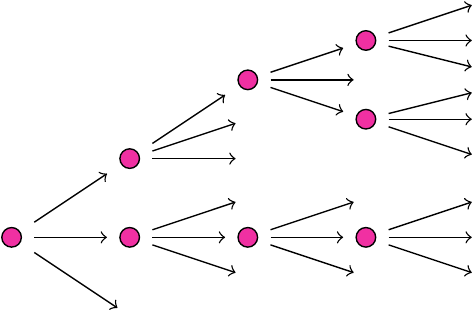}
\caption{This has algebra equal to $\im \, \nu_{3,4}$.\label{fig:t34b}}
\end{subfigure}
\quad
\begin{subfigure}[b]{0.3\textwidth}
\centering
\includegraphics{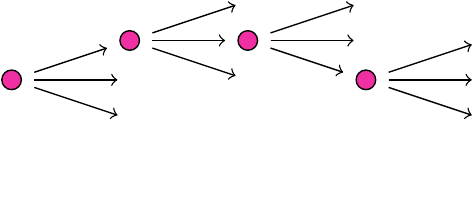}
\caption{Does this staged tree have toric structure?  \label{fig:t34c}}
\end{subfigure}
\caption{Three staged trees in $\mathfrak{T}_{3,4}$. The edge-labels are $\theta_1,\theta_2,\theta_3$,  respectively, read from top to bottom for each vertex.\label{fig:34}}
\end{figure} 
  
\begin{conj}All one-stage tree models have toric structure. \end{conj}
 
We end this section with a discussion of the linear relations appearing in the implicit description of a one-stage tree model. As the following two propositions will show,  unless the structure of the tree is a caterpillar,  the ideal $\kerf$ of one-stage tree models always contains linear relations. 

\begin{prop}\label{prop:linear}
Non-caterpillar one-stage trees contain linear relations in their implicit description. 
\end{prop}

\begin{proof}

Take a vertex $u$ in the non-caterpillar tree $\T$ that has at least two internal vertices, say $u_i$ and $u_j$, as its children. Let $\theta_i,\theta_j$ be the parameters for the transitions from $u$ to $u_i$ and $u_j$, respectively. Take the child  $u_{ij}$ of $u_i$ whose associated edge label is $\theta_j$. Similarly, let $u_{ji}$ be the child of $u_j$ with edge label $\theta_i$.  Then, $\varphi_{\T}(p_{[u_{ij}]})=\varphi_{\T}(p_{[u]})\theta_{i}\theta_j=\varphi_{\T}(p_{[u_{ji}]})$, and $p_{[u_{ij}]}-p_{[u_{ji}]}=0$ is in the kernel of $\varphi_{\T}$. 
\end{proof}

The following result is valid for all staged trees that share a caterpillar structure and not only for the one-stage tree models. We thus prove it for the general case.

\begin{prop}\label{prop:caterpillar}
Caterpillar trees do not contain linear relations in their implicit form.
\end{prop}


\begin{proof}
Take  $\cC$ to be a caterpillar tree. 
Let $\theta_{t,1},\theta_{t,2},\dots,\theta_{t,k_t}$ be the children of internal vertex at the level $t-1$ of the caterpillar tree. To avoid heavy notation, we work with the non-homogenised version.
Each root-to-leaf path has a parametrisation $\theta(1)\cdots \theta(t-1)\cdot \theta_{t,j}$, where $1\leq t\leq d-1$, $1\leq j\leq k_t$, and $\theta_{t,j}\neq \theta(t)$.  
Assume by contradiction that the linear relationship $\sum_{i\in U}a_i\cdot p_i=0$, where $U$ is a subset of leaves in $\cC$ and $a_i$ are real values, is in $\ker \,\varphi_{\cC}$.  Its image under $\varphi_\cC$ has the form \(\sum_{i\in U}a_i\cdot \theta(1)\cdots \theta(t_i-1)\cdot \theta_{t_i,j_i}=0,\) where again $\theta_{t_i,j_i}$ is always different from $\theta(t_i)$. 
 Let $s$ be the  maximum number of edges appearing in all root-to-leaf paths ending at a vertex in  $U$. Take  $U_0$ to be the non-empty subset of paths of length $s$ in $I$. After dividing by $\theta(1)\cdots \theta(s-1)$ and separating all  the new linear terms afterwards to the left of the equation, we find
\begin{equation}\label{eq:fakelinear} 
\sum_{i\in U_0}a_i\cdot \theta_{s,j_i}=-\sum_{i\in U\setminus U_0}a_i\cdot \theta(s)\cdots \theta(t_i-1)\cdot \theta_{t_i,j_i}.
\end{equation} 
In order for the equality \cref{eq:fakelinear}  to hold, the right hand-side must be simplified to a sum of linear forms.  
However, manipulating with sum-to-one conditions can at most  give the linear term $\theta(s)$, which  is different from any of the terms $\theta_{s,j_i}$ to the left of \cref{eq:fakelinear}. 
\end{proof}

\section{Further directions}
\label{sec:discussion}


Bayesian networks are arguably the most famous class of staged tree models. In addition to the staged tree representation, they can be visualised by directed acyclic graphs (DAGs): we refer to~\cite{Collazo.etal.2018} for details on this construction. Those DAGs which are decomposable are also toric~\cite{Geiger.etal.2006}. At the time of writing it is an open question whether all Bayesian networks have toric structure. Here, we use Theorem \ref{thm:extend_toric} to show that the non-decomposable Bayesian networks presented by DAGs in Figure \ref{fig:dags}, opening the path for possible applications to methods developed in this paper in proving that all of them have toric structure. The staged tree representations of the DAGs in \cref{fig:dags} for binary variables are shown in \cref{fig:bayesian} but the patterns of these trees are the same for random variables with an arbitrary number of states. Both staged trees satisfy the conditions of \Cref{thm:extend_toric} and hence have toric structure. Here, the set of internal vertices at distance $3$ from the root serve as $v_1,\dots,v_m$ for the tree corresponding to \cref{fig:bayesian1}. Vertices at distance $2$ from the root play this role in the tree in \cref{fig:bayesian2}. 

\begin{conj}
All Bayesian networks have toric structure. 
\end{conj}

There are staged trees which are not covered by \cref{thm:extend_toric} but can still be proven to have a toric structure using Theorem \ref{thm:toricJ}. In fact we saw such a tree already in Example \ref{ex:minors}. Recall that one requirement in Theorem \ref{thm:toricJ} was that the new variables are mapped to monomials under $\varphi_\T$. This in turn gives the parametrisation of the variety of $\ker \varphi_\T$. In many cases, such as most balanced trees, the ideal $J_\T$ is properly included in $\kerf$.  So, in general it is not sufficient to only prove that $J_\T$ is binomial. 
It can of course happen that $J_\T $ is equal to $\ker \varphi_\T$ as in the example below.

\begin{figure}
\centering
\begin{subfigure}[b]{0.45\linewidth}
\centering
\includegraphics[scale=0.8]{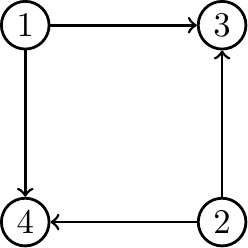}
\caption{Bayesian network number $9$ in Table $1$ of \cite{Garcia.etal.2005}.\label{fig:dag1}}
\end{subfigure}
\qquad
\begin{subfigure}[b]{0.45\linewidth}
\centering
\includegraphics[scale=0.8]{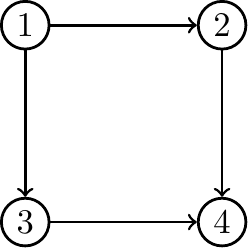}
\caption{Bayesian network number $15$ in Table $1$ of \cite{Garcia.etal.2005}.\label{fig:dag2}}
\end{subfigure}%
\caption{Two non-decomposable DAGs in four random variables, statistically equivalent to the staged trees in \cref{fig:bayesian}.}
\label{fig:dags}
\end{figure}
\begin{figure}
\centering
\begin{subfigure}[b]{0.45\linewidth}
\centering
\includegraphics[scale=0.8]{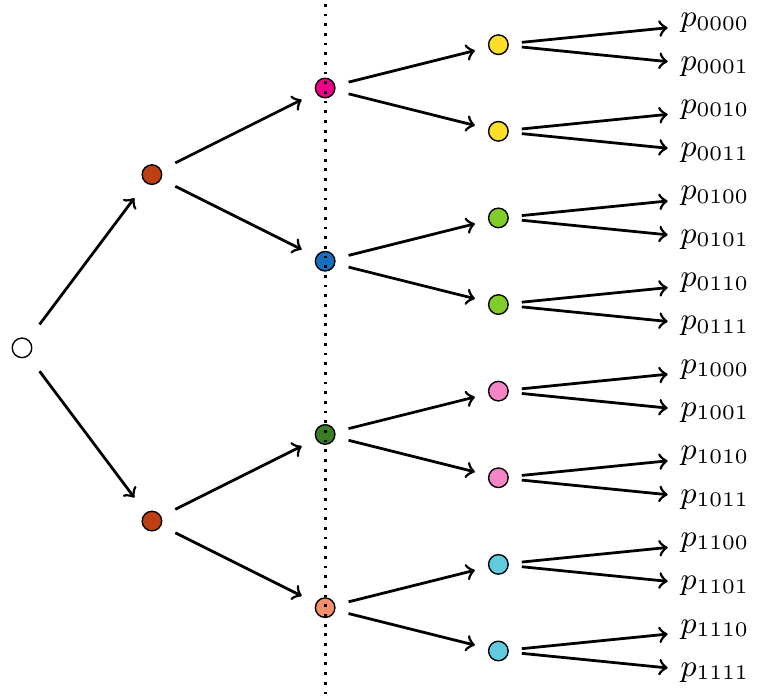}
\caption{A staged tree representing the DAG in \cref{fig:dag1}.\label{fig:bayesian1}}
\end{subfigure}
\qquad
\begin{subfigure}[b]{0.45\linewidth}
\centering
\includegraphics[scale=0.8]{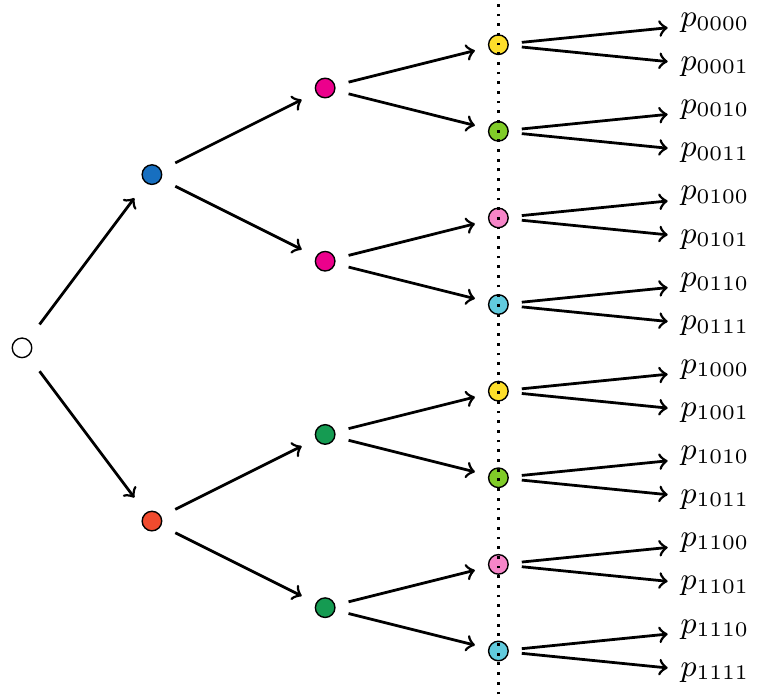}
\caption{A staged tree representing the DAG in \cref{fig:dag2}. \label{fig:bayesian2}}
\end{subfigure}%
\caption{Each  of the subtrees to the left of the dotted lines are balanced, and all the subtrees to the right of the doted line satisfy the subtree-inclusion property. \label{fig:bayesian}}
\end{figure}

On the other side, both the trees in \cref{fig:t34c} and \cref{fig:april} have $J_\T=\ker \varphi_\T$ but none of our endeavours found linear transformations that  give the correct number of distinct entries, and it is as of now unknown whether these models are toric.  

An additional simple idea to find toric structure is a computational approach we present at:
\begin{center}
\url{https://mathrepo.mis.mpg.de/StagedTreesWithToricStructures}.
\end{center}
This uses the computer algebra softwares \texttt{Macaulay2}~\cite{M2} and \texttt{Mathematica}~\cite{Mathematica} to either visually check whether a staged tree's prime ideal has binomial generators or to randomly create linear transformations doing row and column operations on stage matrices as in Example~\ref{ex:nastytoric} until we end up with an isomorphism which reveals underlying toric structure. It is beyond the scope of this paper to develop efficient algorithms for these ideas.

\begin{figure}
\centering
\begin{minipage}[b]{0.45\linewidth}
\centering
\includegraphics{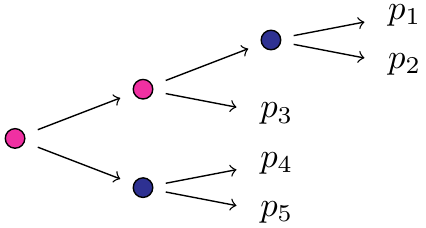}
\caption{A staged tree with toric structure not covered by Theorem \ref{thm:toricJ}. \label{fig:nastytoric}}
\end{minipage}
\qquad
\begin{minipage}[b]{0.45\linewidth}
\centering
\includegraphics{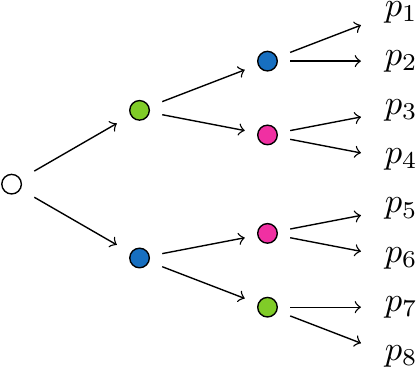}
\caption{Does this staged tree have toric structure? \label{fig:april}}
\end{minipage}%
\end{figure}

As we extend the class of staged tree models with toric structure and do not find counter-examples, we dare to finalise this paper with the wishful  conjecture that all of these have toric structure.

\begin{conj} 
All staged tree models have  toric structure.
\end{conj} 

\section*{Acknowledgements}
The authors would like to thank Eliana Duarte for rewarding discussions on the topics of this paper.

\bibliography{references}

\begin{thebibliography}{10}

\bibitem{AD}
Lamprini Ananiadi and Eliana Duarte.
\newblock Gr{\"o}bner bases for staged trees.
\newblock {\em Alg. Stat.}, 12:1--20, 2021.

\bibitem{Anick}
David~J. Anick.
\newblock On the homology of associative algebras.
\newblock {\em Trans. Amer. Math. Soc.}, 296:641--659, 1986.

\bibitem{bocci2019introduction}
Cristiano Bocci and Luca Chiantini.
\newblock {\em An introduction to algebraic statistics with tensors}, volume~1.
\newblock Springer, 2019.

\bibitem{rpackage}
Federico Carli, Manuele Leonelli, Eva Riccomagno, and Gherardo Varando.
\newblock The \texttt{{R}}-package \texttt{stagedtrees} for structural learning
  of stratified staged trees.
\newblock {\em {Preprint available from \texttt{arXiv:2004.06459[stat.ME]}}},
  2020.

\bibitem{Collazo.etal.2018}
Rodrigo~A. Collazo, Christiane G{\"o}rgen, and Jim~Q. Smith.
\newblock {\em Chain Event Graphs}.
\newblock Computer Science \& Data Analysis Series. Chapman \& Hall, 2018.

\bibitem{cox2011toric}
David~A Cox, John~B Little, and Henry~K Schenck.
\newblock {\em Toric varieties}, volume 124.
\newblock American Mathematical Soc., 2011.

\bibitem{diaconis1998algebraic}
Persi Diaconis, Bernd Sturmfels, et~al.
\newblock Algebraic algorithms for sampling from conditional distributions.
\newblock {\em Annals of statistics}, 26(1):363--397, 1998.

\bibitem{DG}
Eliana Duarte and Christiane G{\"o}rgen.
\newblock Equations defining probability tree models.
\newblock {\em J. Symbolic Comput.}, 99:127--146, 2020.

\bibitem{duarte2021discrete}
Eliana Duarte, Orlando Marigliano, and Bernd Sturmfels.
\newblock Discrete statistical models with rational maximum likelihood
  estimator.
\newblock {\em Bernoulli}, 27(1):135--154, 2021.

\bibitem{Duarte2020AlgebraicGO}
Eliana Duarte and Liam Solus.
\newblock Algebraic geometry of discrete interventional models.
\newblock {\em {Preprint available from \texttt{arXiv:2012.03593[math.ST]}}},
  2020.

\bibitem{Duarte.Solus.2021}
Eliana Duarte and Liam Solus.
\newblock Representation of context-specific causal models with observational
  and interventional data.
\newblock {\em Preprint available from \texttt{arXiv:2101.09271[math.ST]}},
  2021.

\bibitem{EH}
Viviane Ene and J{\"u}rgen Herzog.
\newblock {\em Gr{\"o}bner {B}ases in {C}ommutative {A}lgebra}, volume 130 of
  {\em Graduate {S}tudies in {M}athematics}.
\newblock American Mathematical Society, Providence, RI, 2012.

\bibitem{fienberg2012maximum}
Stephen~E Fienberg and Alessandro Rinaldo.
\newblock Maximum likelihood estimation in log-linear models.
\newblock {\em The Annals of Statistics}, 40(2):996--1023, 2012.

\bibitem{Garcia.etal.2005}
Luis~David Garcia, Michael Stillman, and Bernd Sturmfels.
\newblock Algebraic geometry of {B}ayesian networks.
\newblock {\em J. Symbolic Comput.}, 39(3-4):331--355, 2005.

\bibitem{Geiger.etal.2001}
Dan Geiger, David Heckerman, Henry King, and Christopher Meek.
\newblock Stratified exponential families: graphical models and model
  selection.
\newblock {\em Ann. Statist.}, 29(2):505 -- 529, 2001.

\bibitem{Geiger.etal.2006}
Dan Geiger, Christopher Meek, and Bernd Sturmfels.
\newblock On the toric algebra of graphical models.
\newblock {\em Ann. Statist.}, 34(3):1463--1492, 2006.

\bibitem{Genewein.etal.2021}
Tim Genewein, Tom McGrath, Grégoire Delétang, Vladimir Mikulik, Miljan
  Martic, Shane Legg, and Pedro~A. Ortega.
\newblock Algorithms for causal reasoning in probability trees.
\newblock {\em {Preprint available from \texttt{arXiv:2106.04416[cs.AI]}}},
  2021.

\bibitem{Goergen.etal.2018}
Christiane G{\"o}rgen, Anna Bigatti, Eva Riccomagno, and Jim~Q. Smith.
\newblock Discovery of statistical equivalence classes using computer algebra.
\newblock {\em Internat. J. Approx. Reason.}, 95:167--184, 2018.

\bibitem{Goergen.etal.2020}
Christiane G{\"o}rgen, Manuele Leonelli, and Orlando Marigliano.
\newblock {The curved exponential family of a staged tree}.
\newblock {\em {Preprint available from \texttt{arXiv:2010.15515[math.ST]}}},
  2020.

\bibitem{Goergen.Smith.2018}
Christiane G{\"o}rgen and Jim~Q. Smith.
\newblock Equivalence classes of staged trees.
\newblock {\em Bernoulli}, 24(4A):2676--2692, 2018.

\bibitem{M2}
Daniel~R. Grayson and Michael~E. Stillman.
\newblock Macaulay2, a software system for research in algebraic geometry.
\newblock Available at \url{https://faculty.math.illinois.edu/Macaulay2/}.

\bibitem{Harris}
Joe Harris.
\newblock {\em Algebraic geometry: a first course}, volume 133.
\newblock Springer Science \& Business Media, 2013.

\bibitem{Hochster}
Melvin Hochster.
\newblock Rings of invariants of tori, {C}ohen-{M}acaulay rings generated by
  monomials, and polytopes.
\newblock {\em Ann. of Mathematics}, 96:228--235, 1972.

\bibitem{Keeble.etal.2017}
Claire Keeble, Peter~A. Thwaites, Paul~D. Baxter, Stuart Barber, Roger~C.
  Parslow, and Graham~R. Law.
\newblock Learning through chain event graphs: the role of maternal factors in
  childhood type {I} diabetes.
\newblock {\em American Journal of Epidemiology}, 186(10):1204--1208, 2017.

\bibitem{Leonelli.Varando.2021}
Manuele Leonelli and Gherardo Varando.
\newblock {Context-specific causal discovery for categorical data using staged
  trees}.
\newblock {\em {Preprint available from \texttt{arXiv:2106.04416[stat.ME]}}},
  2021.

\bibitem{Pistone.etal.2001}
Giovanni Pistone, Eva Riccomagno, and Henry~P. Wynn.
\newblock {\em Algebraic statistics}, volume~89 of {\em Monographs on
  Statistics and Applied Probability}.
\newblock Chapman \& Hall/CRC, Boca Raton, FL, 2001.
\newblock Computational commutative algebra in statistics.

\bibitem{Smith.Anderson.2008}
Jim~Q. Smith and Paul~E. Anderson.
\newblock Conditional independence and chain event graphs.
\newblock {\em Artificial Intelligence}, 172(1):42--68, 2008.

\bibitem{Mathematica}
{Wolfram Research{,} Inc.}
\newblock Mathematica, {V}ersion 12.2.
\newblock {C}hampaign, IL, 2020.

\end{thebibliography}

\end{document}